\documentclass[11pt,letterpaper]{amsart}
\usepackage{amsmath}
\usepackage{amssymb}
\usepackage{amsthm}
\usepackage{amsxtra}
\usepackage{graphicx}
\usepackage[all,cmtip]{xy}
\usepackage{hyperref}
\usepackage{enumitem}
\usepackage{color}
\usepackage{blkarray}
\usepackage{mathabx}
\usepackage{microtype}
\usepackage{geometry}
\usepackage{cite}

\newtheorem{theorem}{Theorem}
\newtheorem{lemma}[theorem]{Lemma}
\newtheorem{corollary}[theorem]{Corollary}
\newtheorem{proposition}[theorem]{Proposition}
\theoremstyle{definition}
\newtheorem{definition}[theorem]{Definition}

\theoremstyle{remark}
\newtheorem{remark}[theorem]{Remark}
\numberwithin{equation}{section}
\numberwithin{theorem}{section}
\numberwithin{problem}{section}
\input{tcilatex}
\begin{document}
\title[Non Scale-Invariant Well/Ill-posedness Separation for Boltzmann]{%
Well/ill-posedness bifurcation for the Boltzmann equation with constant
collision kernel}
\author{Xuwen Chen}
\address{Department of Mathematics, University of Rochester, Rochester, NY
14627}
\email{xuwenmath@gmail.com}
\urladdr{http://www.math.rochester.edu/people/faculty/xchen84/}
\author{Justin Holmer}
\address{Department of Mathematics, Brown University, 151 Thayer Street,
Providence, RI 02912}
\email{justin$\underline{\;\,}$holmer@brown.edu}
\urladdr{http://www.math.brown.edu/\symbol{126}holmer/}

\begin{abstract}
We consider the 3D Boltzmann equation with the constant collision kernel. We
investigate the well/ill-posedness problem using the methods from nonlinear
dispersive PDEs.\ We construct a family of special solutions, which are
neither near equilibrium nor self-similar, to the equation, and prove that
the well/ill-posedness threshold in $H^{s}$ Sobolev space is exactly at
regularity $s=1$, despite the fact that the equation is scale invariant at $%
s=\frac{1}{2}$.
\end{abstract}

\maketitle
\tableofcontents


\section{Introduction}

The fundamental Boltzmann equation describes the time-evolution of the
statistical behavior of a thermodynamic system away from a state of
equilibrium in the mesoscopic regime, accounting for both dispersion under
the free flow and dissipation as the result of collisions. So far, the
well-posedness of the Boltzmann equation remains largely open even after the
innovative work \cite{DL89, GS11}, while we are not aware of any
ill-posedness results. The goal of this paper is to investigate the fine
properties of well-posedness versus ill-posedness of the Boltzmann equation
via a scaling point of view, using the latest techniques from the field of
nonlinear dispersive PDEs.

Trying out techniques which work for nonlinear dispersive PDEs on the
Boltzmann equation is not without precursors. For years, there have been
many nice developments \cite%
{AMUXY,AMUXY1,AGPT,Ar,DHWY,Duan1,H1,MS,Yan1,SS1,T1,U1}\ which have hinted at
or used space-time estimates like the Chemin-Lerner spaces or the harmonic
analysis related to nonlinear dispersive PDEs, and many of them have reached
global (strong and mild) solutions if the datum is close enough to the
Maxwellians or satisifes some conditions. In the same period of time, the
theory of nonlinear dispersive PDEs has matured into a stage on which the
working function spaces have been cleanly unified, the well-posedness and
ill-posedness (See, for example, the now well-known work \cite%
{CCT1,CCT2,KPV0,KPV1,KPV,MST1,MST2} and also the survey \cite{NT} and the
references within.) away from the scale invariant spaces have been mostly
settled, and global well-posedness for large solutions at critical scaling
has made significant progress. Of course, the Boltzmann equation is
certainly very different from the nonlinear dispersive PDEs and the nuts and
bolts designed for one, so far, do not fit the other. Interestingly, the
recent series of papers \cite{CDP19a,CDP19b,CDP21} by T. Chen, Denlinger,
and Pavlovi\'{c} suggests that a systematic study of the Boltzmann equation
using tools built for the dispersive PDEs might indeed be possible.

It is well-known that, the Wigner transform turns the kinetic transport
operator $\partial _{t}+v\cdot \nabla _{x}$ into the hyperbolic Schr\"{o}%
dinger operator 
\begin{equation}
i\partial _{t}+\Delta _{y}-\Delta _{y^{\prime }},
\label{op:hyperbolic schrodinger}
\end{equation}%
and hence the Wigner transform turns the Boltzmann equation into a
dispersive equation with (\ref{op:hyperbolic schrodinger}) being the linear
part. Technical problems remain (if they do not get worse) though, since the
form of the collision operators does not improve and derivatives are
complicated under the Wigner transform while the analysis of (\ref%
{op:hyperbolic schrodinger}) was not very developed for a long time. In \cite%
{CDP19a,CDP19b,CDP21}, incorporating the new developments \cite%
{CHPS,CP1,CP2,CP3,CP4,CP5,CT,CPT,C1,C2,C3,CY,CH1,CH2,CH3,CH4,CH5,CH6,CH7,CH8,CH9,CSZ,CSWZ,CPU,KSS,KM1,GSS,HS1,HS2,S,S1,SS}
on the quantum many-body hierarchy dynamics that rely on the analysis of (%
\ref{op:hyperbolic schrodinger}), with some highly nontrivial technical
improvisions, T. Chen, Denlinger, and Pavlovi\'{c} provided an alternate
dispersive PDE based route for proving the local well-posedness of the
Boltzmann equation. The solutions provided in \cite{CDP19a,CDP19b,CDP21}
differ from previous work in the sense that they solve the Boltzmann
equation a.e. instead of everywhere in time. But these solutions are by no
means weak, as the datum-to-solution map is Lipschitz continuous and
provides persistance of regularity.

We follow the lead of \cite{CDP19a,CDP19b,CDP21}, but we study ill-posedness
instead of well-posedness in this paper (and we are not using the Wigner
transform). It is of mathematical and physical interest to prove
well-posedness at the ``optimal" regularity as it would mean all the
nonlinear interactions of the complicated underlying nonlinear equation have
been analyzed and accounted for. Thus knowing in advance where such
``optimality" lies is instructive (if not important). We prove that, the 1st
guess that the well/ill-posedness split at the critical scaling or the 1st
expecation that self-similar solutions would provide the bad solutions, are
in fact wrong, and the ill-posedness starts to happen in the scaling
subcritical regime. Therefore ordinary perturbative methods in proving
well-posedness fail here, and we have to address the full nonlinear equation
with new ideas and we need our estimates to be as sharp as possible.

As we are using dispersive and harmonic analysis techniques, the \emph{%
constant collision kernel} case would be the logical first-go-to context as
it has a particularly clean form for the loss operator in Fourier variables%
\footnote{%
Our $X_{s,b}$ method in fact does not require the Fourier transform of the
collision kernel. See, for example, Appendix \ref{B:BilinearProof}.} and it
would also clarify the role of the dispersive techniques for future
references as there is not a (partial) convolution or angular integral in
the loss term to generate smoothing. Denoting by $f(t,x,v)$ the phase-space
density, we consider here the 3D Boltzmann equation with constant collision
kernel without boundary: 
\begin{equation}
\partial _{t}f+v\cdot \nabla _{x}f=\int_{\mathbb{S}^{2}}\int_{\mathbb{R}%
^{3}}[f(u^{\ast })f(v^{\ast })-f(u)f(v)]\,du\,d\omega \text{ in }\mathbb{R}%
^{1+6}.  \label{E:boltz1}
\end{equation}%
The variables $u,v$ can be regarded as incoming velocities for a pair of
particles, $\omega \in \mathbb{S}^{2}$ is a parameter for the deflection
angle in the collision of these particles, and the outgoing velocities are $%
u^{\ast },v^{\ast }$: 
\begin{equation*}
u^{\ast }=u+[\omega \cdot (v-u)]\omega \text{ and }v^{\ast }=v-[\omega \cdot
(v-u)]\omega
\end{equation*}%
We adopt the usual gain term and loss term shorthands%
\begin{equation*}
Q(f,g)=Q^{+}(f,g)-Q^{-}(f,g)
\end{equation*}%
\begin{equation*}
Q^{+}(f,g)=\int_{\mathbb{S}^{2}}\int_{\mathbb{R}^{3}}f(v^{\ast })g(u^{\ast
})\,du\,d\omega
\end{equation*}%
\begin{equation*}
Q^{-}(f,g)=\int_{\mathbb{S}^{2}}\int_{\mathbb{R}^{3}}f(v)g(u)\,du\,d\omega
=4\pi f(v)\int_{\mathbb{R}^{3}}g(u)\,du
\end{equation*}%
and equation \eqref{E:boltz1} is invariant under the scaling 
\begin{equation}
f_{\lambda }(t,x,v)=\lambda ^{\alpha +2\beta }f(\lambda ^{\alpha -\beta
}t,\lambda ^{\alpha }x,\lambda ^{\beta }v)  \label{E:scaling condition}
\end{equation}%
for any $\alpha ,\beta \in \mathbb{R}$ and $\lambda >0$.

To draw a more direct connection between the analysis of (\ref{E:boltz1})
and the theory of nonlinear dispersive PDEs, we start with the linear part
of (\ref{E:boltz1}), which, upon passing to the inverse Fourier variable $%
v\mapsto \xi ,$ is the symmetric hyperbolic Schr\"{o}dinger equation 
\begin{equation}
i\partial _{t}\tilde{f}+\nabla _{\xi }\cdot \nabla _{x}\tilde{f}=0,
\label{E:01}
\end{equation}%
where $\tilde{f}(t,x,\xi )$ is the inverse Fourier transform \emph{in the
third (velocity) variable only}. Evolution of initial condition $\phi (x,\xi
)$ along (\ref{E:01}) will be denoted as $\tilde{f}=e^{it\nabla _{\xi }\cdot
\nabla _{x}}\phi $. Solutions to (\ref{E:01}) automatically satisfy
Strichartz estimates, as we shall review in Appendix \ref{A:scaling}, that 
\begin{equation}
\left\Vert \tilde{f}\right\Vert _{L_{t\in I}^{q}L_{x\xi }^{p}}\lesssim
\left\Vert \tilde{f}|_{t=0}\right\Vert _{L_{x\xi }^{2}}\text{, provided }%
\frac{2}{q}+\frac{6}{p}=3\,\text{,\ }q\geqslant 2  \label{e:Strichartz}
\end{equation}%
Thus one considers the Sobolev norms defined by 
\begin{equation*}
\Vert \tilde{f}\Vert _{H_{x\xi }^{s,r}}=\Vert \left\langle \nabla
_{x}\right\rangle ^{s}\left\langle \nabla _{\xi }\right\rangle ^{r}\tilde{f}%
\Vert _{L_{x\xi }^{2}}=\Vert \left\langle \nabla _{x}\right\rangle
^{s}\left\langle v\right\rangle ^{r}f\Vert _{L_{xv}^{2}}=\Vert f\Vert
_{L_{v}^{2,r}H_{x}^{s}}.
\end{equation*}%
However, not only is it instinctive to require the same regularity of $x$
and $\xi $ due to the symmetry in $x$ and $\xi $ in (\ref{E:01}), the
scaling invariance (\ref{E:scaling condition}) of (\ref{E:boltz1}) also
suggests\footnote{%
See a disscussion in Appendix \ref{A:scaling}.} that it is natural to take $%
s=r$ and seek to reach the scaling-invariant critical level\footnote{%
Instead of based on scaling invariance of equation, some people define the
critical regularity for the Boltzmann equation at $s=\frac{3}{2}$, the
continuity threshold.} of $s=\frac{1}{2}$, even though the nonlinear part of %
\eqref{E:boltz1} is not symmetric in $x$ and $\xi .$ Thus we will use the
Sobolev norm 
\begin{equation}
\Vert \tilde{f}\Vert _{H_{x\xi }^{s}}=\Vert \left\langle \nabla
_{x}\right\rangle ^{s}\left\langle \nabla _{\xi }\right\rangle ^{s}\tilde{f}%
\Vert _{L_{x\xi }^{2}}=\Vert \left\langle \nabla _{x}\right\rangle
^{s}\left\langle v\right\rangle ^{s}f\Vert _{L_{xv}^{2}}=\Vert f\Vert
_{L_{v}^{2,s}H_{x}^{s}}  \label{e:sobolev norm}
\end{equation}%
and we say (\ref{E:boltz1}) is $\dot{H}_{x\xi }^{\frac{1}{2}}$-invariant or $%
\dot{L}_{v}^{2,\frac{1}{2}}\dot{H}_{x}^{\frac{1}{2}}$-invariant (in scale).

The focus of the paper is on estimates in the Fourier restriction norm
spaces as in \cite{BE1,B1,KM,RR1}, directly associated with the $H_{x\xi
}^{s}$ space and the propagtor $e^{it\nabla _{\xi }\cdot \nabla _{x}}$, that
we now define. We will denote by $\eta $ the Fourier dual variable of $x$
and by $v$ the Fourier dual variable of $\xi $. The function $\hat{f}(\eta
,v)$ denotes the Fourier transform of $\tilde{f}(x,\xi )$ in both $x\mapsto
\eta $ and $\xi \mapsto v$, and is thus the Fourier transform of $f(x,v)$
itself in only $x\mapsto \eta $. The Fourier restriction norm spaces (or $X$
spaces) associated with equation \eqref{E:01} are 
\begin{equation}
\Vert \tilde{f}\Vert _{X_{s,b}}=\Vert \hat{f}(\tau ,\eta ,v)\langle \tau
+\eta \cdot v\rangle ^{b}\langle \eta \rangle ^{s}\langle v\rangle ^{s}\Vert
_{L_{\tau ,\eta ,v}^{2}},  \label{e:X norm}
\end{equation}%
and it is customary to define their finite time restrictions via%
\begin{equation*}
\Vert \tilde{f}\Vert _{X_{s,b}^{T}}=\inf \left\{ \left\Vert F\right\Vert
_{X_{s,b}}:F|_{\left[ -T,T\right] }=f\right\} .
\end{equation*}%
The form of the gain and loss operators in the $(x,\xi )$ variables are: 
\begin{eqnarray}
\tilde{Q}^{+}\left( \tilde{f},\tilde{g}\right) \left( \xi \right) &=&\int_{%
\mathbb{S}^{2}}\tilde{f}(\xi ^{+})\tilde{g}(\xi ^{-})d\omega ,
\label{eqn:bobylev for gain} \\
\tilde{Q}^{-}(\tilde{f},\tilde{g})(\xi ) &=&\tilde{f}(\xi )\tilde{g}(0)
\label{eqn:bobylev for loss}
\end{eqnarray}%
where $\xi ^{+}=\frac{1}{2}(\xi +\left\vert \xi \right\vert \omega )$ and $%
\xi ^{-}=\frac{1}{2}(\xi -\left\vert \xi \right\vert \omega )$, by the
well-known Bobylev identity \cite{B88}.

As there are the $(x,v)$ and $(x,\xi )$ sides of (\ref{E:boltz1}) in this
paper, to make things clear, we recall the definition of a strong solution
and well-posedness.

\begin{definition}
We say $f(t,x,v)$ is a strong $L_{v}^{2,s}H_{x}^{s}$ or $H_{x\xi }^{s}$
solution to (\ref{E:boltz1}) on $[-T,T]$ if $\tilde{f}\in X_{s,\frac{1}{2}%
+}^{T}$ (in particular, $\tilde{f}\in C([-T,T],H_{x\xi }^{s})$ and $f\in
C([-T,T],L_{v}^{2,s}H_{x}^{s})$) and satisfies%
\begin{equation}
\left( i\partial _{t}+\nabla _{\xi }\cdot \nabla _{x}\right) \tilde{f}=%
\tilde{Q}\left( \tilde{f},\tilde{f}\right) ,  \label{E:boltz2}
\end{equation}%
in which both sides are well-defined as $X_{s,-\frac{1}{2}+}^{T}$ functions
(in particular, $L_{\left[ -T,T\right] }^{1}H_{x\xi }^{s}$ or $L_{\left[ -T,T%
\right] }^{1}L_{v}^{2,s}H_{x}^{s}$ functions).
\end{definition}

\begin{definition}
\label{def:wellposedness}We say (\ref{E:boltz1}) is well-posed in $H_{x\xi
}^{s}$ or $L_{v}^{2,s}H_{x}^{s}$ if for each $R>0$, there exists a time $%
T=T(R)>0,$ such that all of the following are satisfied.

\begin{enumerate}
\item[(a)] (Existence) For each $f_{0}\in L_{v}^{2,s}H_{x}^{s}$ with $%
\left\Vert f_{0}\right\Vert _{L_{v}^{2,s}H_{x}^{s}}\leqslant R$. There is a $%
f(t,x,v)$ such that $f(t,x,v)$ is a strong $L_{v}^{2,s}H_{x}^{s}$ or $%
H_{x\xi }^{s}$ solution to (\ref{E:boltz1}) on $[-T,T]$. Moreover, $%
f(t,x,v)\geqslant 0$ if $f_{0}\geqslant 0$.

\item[(b)] (Conditional Uniqueness) Suppose $f$ and $g$ are two strong $%
L_{v}^{2,s}H_{x}^{s}$ or $H_{x\xi }^{s}$ solutions to (\ref{E:boltz1}) on $%
[-T,T]$ with $f|_{t=0}=g|_{t=0}$, then%
\begin{equation*}
f-g=\tilde{f}-\tilde{g}=0\text{ on }[-T,T]
\end{equation*}%
as $H_{x\xi }^{s}$ or $L_{v}^{2,s}H_{x}^{s}$ functions.

\item[(c)] (Uniform Continuity of the Solution Map)\footnote{%
One could replace (c) with the Lipschitz continuity which is usually the
case as well.} Suppose $f$ and $g$ are two strong $L_{v}^{2,s}H_{x}^{s}$ or $%
H_{x\xi }^{s}$ solutions to (\ref{E:boltz1}) on $[-T,T],$ $\forall
\varepsilon >0$, $\exists \delta (\varepsilon )$ independent of $f$ or $g$
such that%
\begin{equation}
\left\Vert f(t)-g(t)\right\Vert _{C([-T,T];{L_{v}^{2,s}H_{x}^{s}}%
)}<\varepsilon \text{ provided that }\left\Vert f(0)-g(0)\right\Vert
_{L_{v}^{2,s}H_{x}^{s}}<\delta (\varepsilon )  \label{eq:uniform continuity}
\end{equation}
\end{enumerate}
\end{definition}

In the context defined above, we have exactly found the separating index
between well/ill-posedness for (\ref{E:boltz1}) to be $s=1$ though (\ref%
{E:boltz1}) is actually $\dot{H}_{x\xi }^{\frac{1}{2}}$-invariant or $\dot{L}%
_{v}^{2,\frac{1}{2}}\dot{H}_{x}^{\frac{1}{2}}$-invariant (in scale).

\begin{theorem}[Main Theorem -- Well/ill-posedness]
\label{T:well-ill-posedness} \qquad \newline
\noindent (1) Equation \eqref{E:boltz1} is locally well-posed in $H_{x\xi
}^{s}$ or $L_{v}^{2,s}H_{x}^{s}$ for $s>1$.

\noindent (2) Equation \eqref{E:boltz1} is ill-posed in $H_{x\xi }^{s}$ or $%
L_{v}^{2,s}H_{x}^{s}$ for $\frac{1}{2}<s<1$ in the following senses:

\begin{enumerate}
\item[2a)] The data-to-solution map is not uniformly continuous and hence
(c) or (\ref{eq:uniform continuity}) in Definition \ref{def:wellposedness}
is violated. In particular, given any $t_{0}\in \mathbb{R}$, for each $M\gg
1 $, there exists a time sequence $\left\{ t_{0}^{M}\right\} _{M}$ such that 
$t_{0}^{M}<t_{0},$ $t_{0}^{M}\nearrow t_{0}$ and two solutions $f^{M}(t)$, $%
g^{M}(t)$ to equation \eqref{E:boltz1} in $\left[ t_{0}^{M},t_{0}\right] $
with $\Vert f^{M}(t_{0}^{M})\Vert _{L_{v}^{2,s}H_{x}^{s}},\Vert
g^{M}(t_{0}^{M})\Vert _{L_{v}^{2,s}H_{x}^{s}}\sim 1,$ such that $f^{M}(t)$, $%
g^{M}(t)$ are initially close at $t=t_{0}^{M}$ 
\begin{equation*}
\Vert f^{M}(t_{0}^{M})-g^{M}(t_{0}^{M})\Vert
_{L_{v}^{2,s}H_{x}^{s}}\leqslant \frac{1}{\ln M}\ll 1,
\end{equation*}%
but become fully separated at $t=t_{0}$ 
\begin{equation*}
\Vert f^{M}(t_{0})-g^{M}(t_{0})\Vert _{L_{v}^{2,s}H_{x}^{s}}\sim 1.
\end{equation*}

\item[2b)] Moreover, there exists a family of solutions having norm
deflation forward in time (and hence norm inflation backward in time). In
particular, given any $t_{0}\in \mathbb{R}$, for each $M\gg 1$, there exists
a solution $f_{\text{M}}$ to equation \eqref{E:boltz1} and a $%
T_{0}=T_{0}(M)<t_{0}$ such that 
\begin{equation*}
\Vert f_{\text{M}}(t_{0})\Vert _{L_{v}^{2,s}H_{x}^{s}}\sim \frac{1}{\ln M}%
\ll 1\text{ but }\Vert f_{\text{M}}(T_{0})\Vert _{L_{v}^{2,s}H_{x}^{s}}\sim 
\frac{M^{\delta }}{\ln M}\gg 1.
\end{equation*}
\end{enumerate}
\end{theorem}

The novelty of Theorem \ref{T:well-ill-posedness} is certainly the
ill-posedness / part (b). Part (a) is essentially included in \cite%
{CDP19a,CDP19b} already and is stated and proved here with different
estimates, namely \eqref{E:loss-2} and \eqref{E:gain-1}.

Due to the usual embedding $H^{s+}\hookrightarrow B_{2,1}^{s}\hookrightarrow
H^{s}$, Theorem \ref{T:well-ill-posedness} also proves the
well/ill-posedness separation at exactly $s=1$ for the Besov spaces (like
the ones in, for example, \cite{HS2,SS1}). The \textquotedblleft bad"
solutions we found are not the usually expected self-similar solutions, they
are actually impolsion / cavity like solutions. See Figure \ref{F:tubes} for
an illustration. 
\begin{figure}[h]
\includegraphics[scale=0.7]{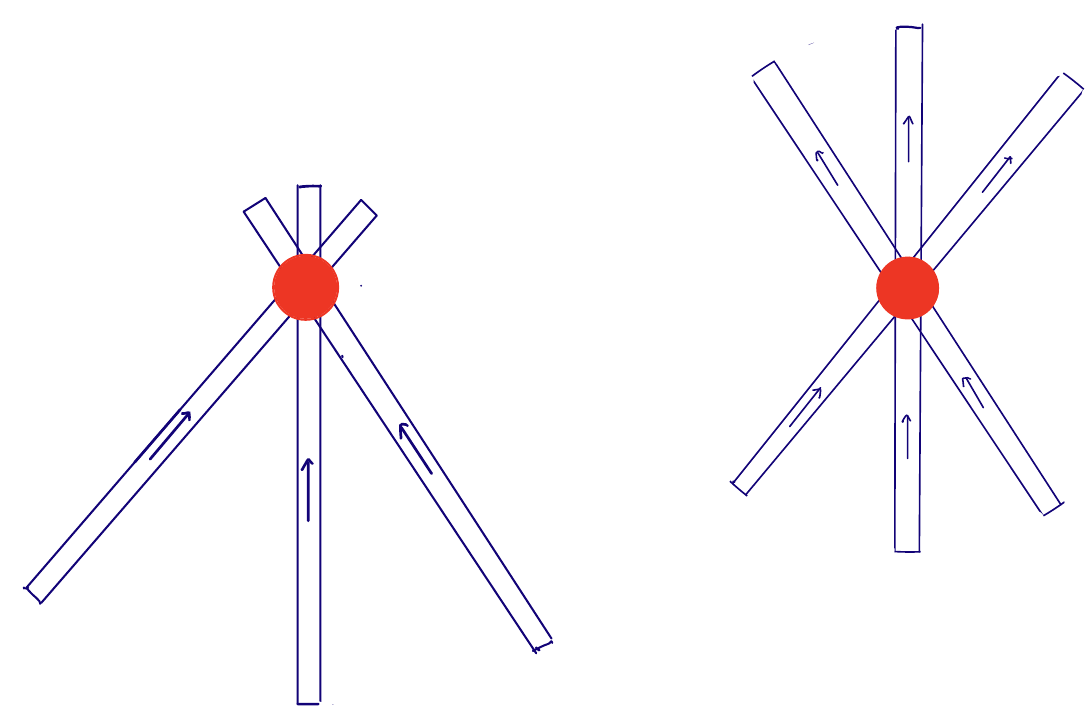}
\caption{Illustration at time $t<0$ and time $t=0$ of the $x$-support of $f_{%
\text{b}}(x,v,t)$ in blue and $f_{\text{r}}(x,v,t)$ in red. The function $f_{%
\text{b}}$ consists of $\sim (MN_{2})^{2}$ tubes; here three typical tubes
are depicted. Each tube moves in the direction of its long dimension. The
ill-posedness part of Theorem \protect\ref{T:well-ill-posedness} exploits
solutions of the form is $f=f_{\text{b}}+f_{\text{r}}+f_{\text{c}}$, where $%
f_{\text{c}}$ is a small correction. As time evolves forward, the loss term
between $f_{\text{b}}$ and $f_{\text{r}}$ drives the amplitude of $f_{\text{r%
}}$ downward exponentially fast.}
\label{F:tubes}
\end{figure}

Moreover, this family of "bad" solutions is uniformly bounded in mass,
variance and kinetic energy. That is, one would still conclude ill-posedness
in $L_{v}^{2,s}H_{x}^{s}$ via Theorem \ref{T:well-ill-posedness} even if one
assume uniformly bounded mass, variance and kinetic energy. (See Remark \ref%
{Remark: finite 2nd moment}.)

Per Theorem \ref{T:well-ill-posedness}, these impolsion / cavity like
solutions are obstacles to the well-posedness for $s<1$ while they do not
pose any problems for $s>1$\textbf{.} We wonder if this phenomenon is
related to the fact that, though more difficult to solve, the Boltzmann
equation is preferred over its various scaling limits, like the compressible
Euler equation for example, in some aspects of engineering like the
aerodynamics at hypersonic speed for example. These \textquotedblleft bad"
solutions arise from the sharpness examples of the bilinear $X_{s,b}$
estimate stated below as Theorem \ref{T:loss}.

\begin{theorem}[loss term bilinear estimate and sharpness example]
\label{T:loss} Let $\theta (t)$ be a smooth function such that $\theta (t)=1$
on $-1\leq t\leq 1$ and $\func{supp}\theta \subset \lbrack -2,2]$. For $\hat{%
f}$ supported in the dyadic regions $|\eta_{1}|\sim M_{1}\geq 1$, $%
|v_{1}|\sim N\geq 1$, and $\hat{g}$ supported in the dyadic regions $|\eta
_{2}|\sim M_{2}\geq 1$, $|v_{2}|\sim N_{2}\geq 1$, 
\begin{equation}
\| \theta(t) \tilde{Q}^{-}(\tilde f,\tilde g) \|_{X_{0,-\frac{1}{2}%
+}}\lesssim \min (M_{1},M_{2})N_{2}B_{M_{1},M_{2}}B_{N,N_{2}}\Vert \tilde{f}%
\Vert _{X_{0,\frac{1}{2}+}}\Vert \tilde{g}\Vert _{X_{0,\frac{1}{2}+}}
\label{E:loss-1}
\end{equation}%
where $B_{M_{1},M_{2}}$ and $B_{N,N_{2}}$ are the one-sided bilinear gain
factors 
\begin{equation*}
B_{M_{1},M_{2}}=%
\begin{cases}
\left( \frac{M_{1}}{M_{2}}\right) ^{1/2} & \text{if }M_{1}\leq M_{2} \\ 
1 & \text{if }M_{1}\geq M_{2}%
\end{cases}%
\,,\qquad B_{N,N_{2}}=%
\begin{cases}
\left( \frac{N_{2}}{N}\right) ^{1/2} & \text{if }N_{2}\leq N \\ 
1 & \text{if }N_{2}\geq N%
\end{cases}%
\end{equation*}%
It follows that for any $s>1$, for any $f$ and $g$ (without frequency
support restrictions), 
\begin{equation}
\| \theta(t) \tilde Q^{-}(\tilde f,\tilde g)\|_{X_{s,-\frac{1}{2}+}}\lesssim
\Vert \tilde{f}\Vert _{X_{s,\frac{1}{2}+}}\Vert \tilde{g}\Vert_{X_{s,\frac{1%
}{2}+}}  \label{E:loss-2}
\end{equation}
There exist functions $f$ and $g$ with $\hat f$ supported in the dyadic
regions $|\eta_1|\sim M_1$, $|v_1|\sim N$, and $\hat g$ supported in the
dyadic regions $|\eta_2|\sim M_2$, $|v_2|\sim N_2$, such that 
\begin{equation*}
\max(N,1)\ll N_2 \, \quad M_1 \geq 1 \,, \quad M_2 \geq 1
\end{equation*}
and 
\begin{equation}  \label{E:loss-sharp-1}
\|\theta(t) \tilde Q^-(\tilde f,\tilde g) \|_{X_{0,-b^{\prime }}} \gtrsim
\min(M_1,M_2) N_2 B_{M_1M_2} \|\tilde f\|_{X_{0,b_1}} \|\tilde
g\|_{X_{0,b_2}}
\end{equation}
for all $b^{\prime }, b_1, b_2 \in \mathbb{R}$.
\end{theorem}

We emphasize that the bilinear gain factors, typically present in these
types of estimates, are both only \emph{one-sided} in this case. This comes
from the fact that the loss operator $Q^{-}$ is highly nonsymmetric in $f$
and $g$, and not of convolution type in the variable $v$. Estimate (\ref%
{E:loss-2}), which is \eqref{E:loss-1} without the bilinear gain factors,
can be proved rather quickly using the Strichartz estimates, but to see the
``bad" solutions, one has to study (\ref{E:loss-1}) and how it becomes
saturated.

In contrast to the loss term, the gain term has bilinear estimates which
almost match the critical scaling $s=\frac{1}{2}$.

\begin{theorem}
\label{T:gain} Let $\theta (t)$ be a smooth function such that $\theta (t)=1$
on $-1\leq t\leq 1$ and $\func{supp}\theta \subset \lbrack -2,2]$. For all $%
s>\frac{1}{2}$, for any $f$ and $g$ (without frequency support
restrictions), 
\begin{equation}
\Vert \theta (t)\tilde{Q}^{+}(\tilde{f},\tilde{g})\Vert _{X_{s,-\frac{1}{2}%
+}}\lesssim \Vert \tilde{f}\Vert _{X_{s,\frac{1}{2}+}}\Vert \tilde{g}\Vert
_{X_{s,\frac{1}{2}+}}  \label{E:gain-1}
\end{equation}
\end{theorem}

The $s=\frac{1}{2}$ case of (\ref{E:gain-1}) might also be correct, but it
cannot change the ill-posedness facts even if it is, due to the loss term.
That is, in low regularity settings, the gain term is like an error,
compared to the loss term. Or in other words, the long expected cancelation
between the gain and loss terms does not happen. Such a disparity between
the gain term and the loss term ultimately creates Theorem \ref%
{T:well-ill-posedness}.

\subsection{Organization of the Paper}

The novelty of this paper lies in Theorem \ref{T:well-ill-posedness} which
establishes the well/ill-posedness seperation and proves the ill-posedness.
Because the ill-posedness happens at scaling subcritical regime, ordinary
scaling or perturbative methods for scaling critical or supercritical
regimes fail here. We thus need new ideas addressing the full nonlinear
equation. The 1st step would be getting sharp estimates.

In \S \ref{S:loss}, we prove \eqref{E:loss-2} in Theorem \ref{T:loss} by
appealing to the Strichartz estimates (reviewed in Appendix \ref{A:scaling})
and prove the sharpness of (\ref{E:loss-1}) from which the illposedness
originates. In \S \ref{S:gain}, we prove the gain term estimates in Theorem %
\ref{T:gain} by appealing to a H\"{o}lder type estimate in \cite{AC10}
together with a Littlewood-Paley decomposition.\footnote{%
We end up simplifying the well-posedness estimates in \cite{CDP19a,CDP19b} a
little bit and reducing the regularity requirement for the gain term, as
co-products.} For completeness, we apply in the short section \S \ref%
{S:well-posedness}, the loss term estimate in Theorem \ref{T:loss} and the
gain term estimate in Theorem \ref{T:gain} to prove the well-posedness part
of Theorem \ref{T:well-ill-posedness}.

We prove the illposedness part of Theorem \ref{T:well-ill-posedness} in \S %
\ref{S:ill-posedness}. We work on the more usual $(x,v)$ side of %
\eqref{E:boltz1} throughout \S \ref{S:ill-posedness} but it is still very
much based on $X_{s,b}$ theory. As the proof involves many delicate
computations, we 1st provide a heuristic using the sharpness example of (\ref%
{E:loss-1}), suitably scaled, in \S \ref{s:motivate}, then we present the
approximate solutions / ansatz $f_{a}$ we are going to use in \S \ref%
{sec:precise forumulation of bad}.\footnote{%
Both of the examples in \S \ref{sec:sharpness} and \S \ref{sec:precise
forumulation of bad} maximize (\ref{E:loss-1}). The one in \S \ref%
{sec:sharpness} is more straight forward for the purpose of maximizing (\ref%
{E:loss-1}).} We use a specialized perturbative argument, built around $%
f_{a} $, to prove that there is an exact solution to \eqref{E:boltz1} which
is mostly $f_{a}$. We put down the tools in \S \ref{sec:property of f_a} to 
\S \ref{sec:z-norm bilinear} to conclude the perturbative argument in \S \ref%
{sec:perturbative}. During the estimate of the error terms, the proof
involves many geometric techniques, like the ones in $X_{s,b}$ theory (See,
for example, \cite{CDKS,KMRemarks})$,$ on the fine nonlinear interactions. A
good example would be \S \ref{sec:Q+-(f_b,f_b)}. Such a connection between
the analysis of (\ref{E:boltz1}) and the dispersive equations might be
highly nontrivial and deserves further investigations. We then conclude the
ill-posedness and the norm deflation in Corollaries \ref{C:norm_deflation}
and \ref{C:data_to_sol}. Our method is general and can be pushed further:
see \cite{CSZ2} for a generalization to well/ill-posedness separation with
general kernels, and then \cite{CSZ3} for a sharp global well-posedness
developed from this paper. Moreover, the derivation of the Boltzmann
equation\ is also related to the well/ill-posedness threshold we found in
the sense of \cite{CH10}.

After the proof of the main theorem in \S \ref{S:ill-posedness}, we include
in Appendix \ref{A:scaling} a discussion of the scaling properties of %
\eqref{E:boltz1} and review the Strichartz estimates, and in Appendix \ref%
{B:BilinearProof}, a full proof of a more manageable version of (\ref%
{E:loss-1}) capturing its salient features. We have not used (\ref{E:loss-1}%
) in this paper but its sharpness has motivated the construction of the
example generating the ill-posedness. Some additional comments are included
in the arxiv.org version of this paper.

\medskip

\begin{flushleft}
\textbf{Acknowledgement}. The authors would like to thank Thomas Chen, Ryan
Denlinger, Yan Guo, Kenji Nakanishi, Nata\v{s}a Pavlovi\'{c}, and Tong Yang
for encouraging and delightful discussions and the anonymous referee for
kind and helpful comments regarding this work. The 1st author was supported
in part by NSF grant DMS-2005469 and a Simons Fellowship numbered 916862.
The 2nd author was partially supported by NSF grant DMS-2055072.
\end{flushleft}

\medskip

\begin{flushleft}
\textbf{Conflict of Interest Statement}. The authors declare there are no
conflicts of interest.
\end{flushleft}

\section{Loss term bilinear estimate and its sharpness\label{S:loss}}

In this section, we prove (\ref{E:loss-2}), the simpler version of the full
optimal bilinear estimate (\ref{E:loss-1}), and establish the sharpness of (%
\ref{E:loss-1}). The proof of (\ref{E:loss-1}) is highly technical as it
needs a lengthy and direct analysis of frequency interactions involving a
conic decomposition of $v$-space and a dual wedge decomposition of $\eta $%
-space. Though the sharpness of (\ref{E:loss-1}) has made the illposedness
possible, we have not used it in this paper. Thus, we prove, for
completeness, a version of (\ref{E:loss-1}) which allows a shorter proof in
Appendix \ref{B:BilinearProof}. The logical order, as it has happened, was,
proving (\ref{E:loss-1}) $\rightarrow $ finding examples to saturate (\ref%
{E:loss-1}) $\rightarrow $ ill-posedness, and one would not have come up
with the ill-posedness without the 1st step.

\subsection{Proof of (\protect\ref{E:loss-2})}

It suffices to prove%
\begin{equation*}
\left\Vert \tilde{Q}^{-}(\tilde{f},\tilde{g})\right\Vert _{L_{t}^{2}H_{x\xi
}^{s}}\lesssim \Vert \tilde{f}\Vert _{X_{s,\frac{1}{2}+}}\Vert \tilde{g}%
\Vert _{X_{s,\frac{1}{2}+}}
\end{equation*}%
for $s>1$. Recall (\ref{eqn:bobylev for loss}), 
\begin{equation*}
\left\Vert \tilde{Q}^{-}(\tilde{f},\tilde{g})\right\Vert _{L_{t}^{2}H_{x\xi
}^{s}}=\left\Vert \left\langle \nabla _{x}\right\rangle ^{s}\left(
\left\langle \nabla _{\xi }\right\rangle ^{s}\tilde{f}(t,x,\xi )\tilde{g}%
(t,x,0)\right) \right\Vert _{L_{t}^{2}L_{x\xi }^{2}}
\end{equation*}%
Splitting into the two cases in which $\tilde{f}$ or $\tilde{g}$ has the
dominating frequency, we have%
\begin{eqnarray*}
\left\Vert \tilde{Q}^{-}(\tilde{f},\tilde{g})\right\Vert _{L_{t}^{2}H_{x\xi
}^{s}} &\lesssim &\left\Vert \left( \left\langle \nabla _{x}\right\rangle
^{s}\left\langle \nabla _{\xi }\right\rangle ^{s}\tilde{f}(t,x,\xi )\right) 
\tilde{g}(t,x,0)\right\Vert _{L_{t}^{2}L_{x\xi }^{2}} \\
&&+\left\Vert \left( \left\langle \nabla _{\xi }\right\rangle ^{s}\tilde{f}%
(t,x,\xi )\right) \left\langle \nabla _{x}\right\rangle ^{s}\tilde{g}%
(t,x,0)\right\Vert _{L_{t}^{2}L_{x\xi }^{2}}
\end{eqnarray*}%
Apply H\"{o}lder,%
\begin{eqnarray*}
\left\Vert \tilde{Q}^{-}(\tilde{f},\tilde{g})\right\Vert _{L_{t}^{2}H_{x\xi
}^{s}} &\lesssim &\left\Vert \left\langle \nabla _{x}\right\rangle
^{s}\left\langle \nabla _{\xi }\right\rangle ^{s}\tilde{f}(t,x,\xi
)\right\Vert _{L_{t}^{\infty }L_{x\xi }^{2}}\left\Vert \tilde{g}(t,x,\xi
)\right\Vert _{L_{t}^{2}L_{x\xi }^{\infty }} \\
&&+\left\Vert \left\langle \nabla _{\xi }\right\rangle ^{s}\tilde{f}(t,x,\xi
)\right\Vert _{L_{t}^{\infty }L_{x}^{6}L_{\xi }^{2}}\left\Vert \left\langle
\nabla _{x}\right\rangle ^{s}\tilde{g}(t,x,\xi )\right\Vert
_{L_{t}^{2}L_{x}^{3}L_{\xi }^{\infty }}
\end{eqnarray*}%
Use $s>1$ and apply Sobolev,

\begin{equation*}
\left\Vert \tilde{Q}^{-}(\tilde{f},\tilde{g})\right\Vert _{L_{t}^{2}H_{x\xi
}^{s}}\lesssim \left\Vert \left\langle \nabla _{x}\right\rangle
^{s}\left\langle \nabla _{\xi }\right\rangle ^{s}\tilde{f}(t,x,\xi
)\right\Vert _{L_{t}^{\infty }L_{x\xi }^{2}}\left\Vert \left\langle \nabla
_{x}\right\rangle ^{s}\left\langle \nabla _{\xi }\right\rangle ^{s}\tilde{g}%
(t,x,\xi )\right\Vert _{L_{t}^{2}L_{x\xi }^{3}}
\end{equation*}%
Apply Strichartz (\ref{e:Strichartz}), we have%
\begin{equation*}
\left\Vert \tilde{Q}^{-}(\tilde{f},\tilde{g})\right\Vert _{L_{t}^{2}H_{x\xi
}^{s}}\lesssim \Vert \tilde{f}\Vert _{X_{s,\frac{1}{2}+}}\Vert \tilde{g}%
\Vert _{X_{s,\frac{1}{2}+}}
\end{equation*}%
as claimed.

\subsection{Proof of Sharpness\label{sec:sharpness}}

At this point, we turn to the sharpness example. We will consider $|\eta
_{2}|\sim M_{2}$ and $|v_{2}|\sim N_{2}$, with $M_{2}\gg 1$ and $N_{2}\gg 1$
dyadic. On the unit sphere, lay down a grid of $J\sim M_{2}^{2}N_{2}^{2}$
points $\{e_{j}\}_{j=1}^{J}$, where the points $e_{j}$ are roughly equally
spaced and each have their own neighborhood of unit-sphere surface area $%
\sim M_{2}^{-1}N_{2}^{-1}$. Let $P_{e_{j}}$ denote the orthogonal projection
onto the 1D subspace spanned by $e_{j}$. Let $P_{e_{j}}^{\perp }$ denote the
orthogonal projection onto the 2D subspace $\func{span}\{e_{j}\}^{\perp }$.


For each $e_j$, let $B_j$ denote, in $v_2$ space, the conic neighborhood of $%
e_j$ obtained by taking all radial rays passing though the patch of surface
area $M_2^{-1}N_2^{-1}$ on the unit sphere around the vector $e_j$. We can
describe this as the cone with vertex $e_j$ and angular aperture $%
M_2^{-1}N_2^{-1}$. When this cone, in $v_2$-space, intersects the dyadic
annulus at radius $\sim N_2$, it creates a shape that is approximately a
cube with long side $N_2$ and two shorter dimensions each of length $%
M_2^{-1} $. For convenience we think of it as a cube with shape $%
M_2^{-1}\times M_2^{-1}\times N_2$, with the line through $e_j$ passing
through the axis and parallel to the long side.

For each $e_{j}$, in $\eta _{2}$ space, let $A_{j}$ consist of all vectors $%
\eta _{2}$ that are \textquotedblleft nearly
perpendicular\textquotedblright\ to $e_{j}$, in the sense that the cosine of
the angle between $e_{j}$ and $\eta _{2}/|\eta _{2}|$ is 
\begin{equation*}
\cos (e_{j},\frac{\eta _{2}}{|\eta _{2}|})\lesssim \frac{1}{M_{2}N_{2}}
\end{equation*}%
Since $\sin (90^{\circ }-\theta )=\cos (\theta )$, this means that the angle
between $e_{j}$ and any vector $\eta _{2}\in A_{j}$ is within $%
M_{2}^{-1}N_{2}^{-1}$ of $90$ degrees. This property then clearly extends to
replacing $e_{j}$ with any vector in $B_{j}$. We can then visualize $B_{j}$,
projected onto the unit sphere, as being obtained by taking the plane
perpendicular to $e_{j}$, calling that the \textquotedblleft
equator\textquotedblright , and then taking the band that is width $%
M_{2}^{-1}N_{2}^{-1}$ around that equator. When $A_{j}$ is intersected with
the dyadic annulus at radius $\sim M_{2}$, this set looks like a thickened
plane of width $N_{2}^{-1}$, so it has approximate dimensions $M_{2}\times
M_{2}\times N_{2}^{-1}$, where the $M_{2}\times M_{2}$ planar part is
perpendicular to $e_{j}$. ($A_{j}$ looks like the so-called bevelled washer.)

The example proving \eqref{E:loss-sharp-1} in Theorem \ref{T:loss} is
produced as follows. Let $\hat\chi(\eta)$ (or $\hat\chi(v)$) be a smooth
nonnegative compactly supported function such that $\hat \chi(0)=1$. We
present the functions $\tilde \phi$ and $\tilde \psi$ as normalized in $%
L^2_xL^2_\xi$.

\begin{equation*}
\hat{\phi}(\eta _{1},v)=\frac{1}{M_{1}^{3/2}N^{3/2}}\hat{\chi}(\frac{\eta
_{1}}{M_{1}})\hat{\chi}(\frac{v}{N})
\end{equation*}%
\begin{equation*}
\hat{\psi}(\eta _{2},v_{2})=\frac{1}{M_{2}N_{2}}\sum_{j=1}^{J}\hat{\chi}%
\left( \frac{P_{e_{j}}^{\perp }\eta _{2}}{M_{2}}\right) \hat{\chi}%
(N_{2}P_{e_{j}}\eta _{2})\hat{\chi}(M_{2}P_{e_{j}}^{\perp }v_{2})\hat{\chi}(%
\frac{P_{e_{j}}v_{2}}{N_{2}})
\end{equation*}%
where $J\sim M_{2}^{2}N_{2}^{2}$. Since we are proving a lower bound, we can
select any normalized dual function $\tilde{\zeta}\in L_{x}^{2}L_{\xi }^{2}$%
. We select 
\begin{equation*}
\hat{\zeta}(\eta ,v)=\frac{1}{\max (M_{1},M_{2})^{3/2}N^{3/2}}\hat{\chi}(%
\frac{\eta }{\max (M_{1},M_{2})})\hat{\chi}(\frac{v}{N})
\end{equation*}%
Then let 
\begin{equation}
I=\int_{\eta }\int_{\eta _{2}}\int_{v}\int_{v_{2}}\hat{\theta}(-\eta
_{2}\cdot (v-v_{2}))\,\hat{\phi}(\eta -\eta _{2},v)\,\hat{\psi}(\eta
_{2},v_{2})\,\overline{\hat{\zeta}(\eta ,v)}\,dv_{2}\,dv\,d\eta _{2}\,d\eta
\label{E:sharp01}
\end{equation}%
and we aim to prove 
\begin{equation*}
I\gtrsim \min (M_{1},M_{2})N_{2}B_{M_{1}M_{2}}\Vert \hat{\phi}\Vert
_{L_{\eta _{1}v}^{2}}\Vert \hat{\psi}\Vert _{L_{\eta _{2}v_{2}}^{2}}\Vert 
\hat{\zeta}\Vert _{L_{\eta v}^{2}}
\end{equation*}%
We are restricting to $M_{1}\geq 1$, $M_{2}\geq 1$ and $N\leq M^{-1}\ll 1\ll
N_{2}$ but otherwise do not assume any relationship between $M_{1}$ and $%
M_{2}$. The argument can be extended to weaken the restriction on $N$ to
only $N\ll N_{2}$.

To carry out the computation, we need to introduce $C_{k}$, subsets of $\eta
_{2}$ space, which are cones\footnote{%
These cone-like decompositions are widely used in the analysis of dispersive
equations. See, for example, \cite{CDKS,KMRemarks,Ta1}.} with angular
aperture $M_{2}^{-1}N_{2}^{-1}$ with vertex vector $e_{k}$. (That is, $C_{k}$
for $k=1,\ldots ,M_{2}^{2}N_{2}^{2}$ is the same type of decomposition of $%
\eta _{2}$ space as the decomposition $B_{j}$, $j=1,\ldots
,M_{2}^{2}N_{2}^{2}$ of $v_{2}$-space, except that in the case of $\{C_{k}\}$%
, the cones intersect the $M_{2}$ dyad, and in the case $\{B_{j}\}$, the
cones intersect the $N_{2}$ dyad.). The cone $C_{k}$ intersects the $M_{2}$
dyad in $\eta _{2}$ space in a tube of geometry $N_{2}^{-1}\times
N_{2}^{-1}\times M_{2}$, whereas the cone $B_{j}$ intersects the $N_{2}$
dyad in $v_{2}$ space in a tube of geometry $M_{2}^{-1}\times
M_{2}^{-1}\times N_{2}$.

For each $j$ and corresponding direction $e_j$, there are $\sim M_2N_2$
directions $e_k$ that are perpendicular to $e_j$ -- let us call this set $%
K(j)$. Then we have that for each $j$, the thickened planes/beveled washers $%
A_j$ are: 
\begin{equation*}
A_j = \bigcup_{k\in K(j)} C_k
\end{equation*}
We can now look at this in the other direction as well -- fix $k$ and
corresponding direction $e_k$. There are $\sim M_2N_2$ directions $e_j$ that
are perpendicular to $e_k$ -- call this set of indices $J(k)$. The union of
these becomes a thickened plane/beveled washer in $v_2$ space.

Now we carry out the integral in \eqref{E:sharp01} by first noting that
since $N\ll N_{2}$, $v_{2}-v$ can effectively be replaced by $v_{2}$ in $%
\hat{\theta}(\eta _{2}\cdot (v_{2}-v))$. This allows us to simply carry out
the $v$ and $v_{2}$ integration after enforcing the orthogonality of $v_{2}$
and $\eta _{2}$ through restriction to appropriate sets $B_{j}$ and $C_{k}$,
respectively. 
\begin{equation}
I=\int_{\eta }\int_{\eta _{2}}\left( \int_{v}\hat{\phi}(\eta -\eta _{2},v)%
\hat{\zeta}(\eta ,v)\,dv\sum_{k}\boldsymbol{1}_{C_{k}}(\eta _{2})\sum_{j\in
J(k)}\int_{v_{2}}\boldsymbol{1}_{B_{j}}(v_{2})\hat{\psi}(\eta
_{2},v_{2})\,dv_{2}\right) \,d\eta _{2}\,d\eta  \label{E:sharp02}
\end{equation}%
The components inside (\ref{E:sharp02}) are 
\begin{equation*}
\int_{v}\hat{\phi}(\eta -\eta _{2},v)\hat{\zeta}(\eta ,v)\,dv=\frac{1}{%
M_{1}^{3/2}\max (M_{1},M_{2})^{3/2}}\hat{\chi}(\frac{\eta -\eta _{2}}{M_{1}})%
\hat{\chi}(\frac{\eta }{\max (M_{1},M_{2})})
\end{equation*}%
and 
\begin{equation}
\sum_{k}\sum_{j\in J(k)}\boldsymbol{1}_{C_{k}}(\eta _{2})\int_{v_{2}}%
\boldsymbol{1}_{B_{j}}(v_{2})\hat{\psi}(\eta _{2},v_{2})\,dv_{2}=\frac{1}{%
M_{2}^{3}}\sum_{k}\sum_{j\in J(k)}\boldsymbol{1}_{C_{k}}(\eta _{2})\hat{\chi}%
\left( \frac{P_{e_{j}}^{\perp }\eta _{2}}{M_{2}}\right) \hat{\chi}%
(N_{2}P_{e_{j}}\eta _{2})  \label{eqn:component2}
\end{equation}%
For (\ref{eqn:component2}), consider that $\hat{\chi}\left( \frac{%
P_{e_{j}}^{\perp }\eta _{2}}{M_{2}}\right) \hat{\chi}(N_{2}P_{e_{j}}\eta
_{2})$ is basically $\boldsymbol{1}_{A_{j}}(\eta _{2})$, but since $e_{k}$
and $e_{j}$ are perpendicular, it follows that $C_{k}\subset A_{j}$ and $%
\boldsymbol{1}_{C_{k}}(\eta _{2})\boldsymbol{1}_{A_{j}}(\eta _{2})=%
\boldsymbol{1}_{C_{k}}(\eta _{2})$. Thus we continue the evaluation as 
\begin{align*}
\hspace{0.3in}& \hspace{-0.3in}\sum_{k}\sum_{j\in J(k)}\boldsymbol{1}%
_{C_{k}}(\eta _{2})\int_{v_{2}}\boldsymbol{1}_{B_{j}}(v_{2})\hat{\psi}(\eta
_{2},v_{2})\,dv_{2}=\frac{1}{M_{2}^{3}}\sum_{k}\sum_{j\in J(k)}\boldsymbol{1}%
_{C_{k}}(\eta _{2}) \\
& =\frac{1}{M_{2}^{3}}\sum_{k}\func{card}J(k)\boldsymbol{1}_{C_{k}}(\eta
_{2})=\frac{N_{2}}{M_{2}^{2}}\sum_{k}\boldsymbol{1}_{C_{k}}(\eta _{2})=\frac{%
N_{2}}{M_{2}^{2}}\hat{\chi}(\frac{\eta _{2}}{M_{2}})
\end{align*}%
where $\func{card}(A)$ of a set $A$ means the cardinality of set $A$. That
is, since $\{C_{k}\}$ just partitions the whole $M_{2}$ dyad, we end up with
simply the projection onto the whole dyad. Plugging these results back into %
\eqref{E:sharp02}, 
\begin{equation*}
I=\int_{\eta }\int_{\eta _{2}}\frac{1}{M_{1}^{3/2}\max (M_{1},M_{2})^{3/2}}%
\hat{\chi}(\frac{\eta -\eta _{2}}{M_{1}})\hat{\chi}(\frac{\eta }{\max
(M_{1},M_{2})})\frac{N_{2}}{M_{2}^{2}}\hat{\chi}(\frac{\eta _{2}}{M_{2}}%
)\,d\eta _{2}\,d\eta
\end{equation*}%
Now we revert back to the spatial side via Plancherel: 
\begin{equation*}
I=M_{1}^{3/2}\max (M_{1},M_{2})^{3/2}M_{2}N_{2}\int_{x}\chi (M_{1}x)\chi
(\max (M_{1},M_{2})x)\chi (M_{2}x)\,dx
\end{equation*}%
The function $\chi (\max (M_{1},M_{2})x)$ is redundant, and the $x$-integral
evaluates to $\max (M_{1},M_{2})^{-3}$. Thus 
\begin{equation*}
I=\frac{M_{1}^{3/2}M_{2}N_{2}}{\max (M_{1},M_{2})^{3/2}}=\min
(M_{1},M_{2})N_{2}B_{M_{1}M_{2}}
\end{equation*}%
which is precisely the lower bound in the statement of Theorem \ref{T:loss}.

\section{Gain term bilinear estimate\label{S:gain}}

We prove Theorem \ref{T:gain} in this section. We will need the following
lemma.

\begin{lemma}[H\"{o}lder Inequality for $\tilde{Q}^{+}$]
\label{Lemma:HolderLikeInequalityForGain}Let $\frac{1}{p}+\frac{1}{q}=\frac{1%
}{r}$ where $p,q>\frac{3}{2}$ and let $b(\frac{\xi }{\left\vert \xi
\right\vert },\omega )\in C\left( \mathbb{S}^{2}\times \mathbb{S}^{2}\right) 
$, we then have the estimate that 
\begin{equation*}
\left\Vert \int_{\mathbb{S}^{2}}\tilde{f}(\xi ^{+})\tilde{g}(\xi ^{-})b(%
\frac{\xi }{\left\vert \xi \right\vert },\omega )d\omega \right\Vert
_{L_{\xi }^{r}}\lesssim \left\Vert \tilde{f}\right\Vert _{L_{\xi
}^{p}}\left\Vert \tilde{g}\right\Vert _{L_{\xi }^{q}}
\end{equation*}
\end{lemma}

\begin{proof}
This is actually a small modification of a special case of \cite[Theorem 1]%
{AC10}. Their proof actually works for all uniformly bounded $b$ with little
modification. We omit the details here.
\end{proof}

Denote the Littlewood-Paley projector of $x/\xi $ at frequency $N/M$ by $%
P_{N}^{x}/P_{M}^{\xi }$. We will use Lemma \ref%
{Lemma:HolderLikeInequalityForGain} in the following format.

\begin{lemma}
\label{Lemma:FreqHolderForGain} \ Let $\frac{1}{p}+\frac{1}{q}=\frac{1}{r}$
where $p,q>\frac{3}{2}$, we then have 
\begin{eqnarray*}
\left\Vert P_{N}^{x}P_{M}^{\xi }\tilde{Q}^{+}\left(
P_{N_{1}}^{x}P_{M_{1}}^{\xi }\tilde{f},P_{N_{2}}^{x}P_{M_{2}}^{\xi }\tilde{g}%
\right) \right\Vert _{L_{\xi }^{r}} &\lesssim &\min (1,\frac{\max
(N_{1},N_{2})}{N})\min (1,\frac{\max (M_{1},M_{2})}{M}) \\
&&\times \left\Vert P_{N_{1}}^{x}P_{M_{1}}^{\xi }\tilde{f}\right\Vert
_{L_{\xi }^{p}}\left\Vert P_{N_{2}}^{x}P_{M_{2}}^{\xi }\tilde{g}\right\Vert
_{L_{\xi }^{q}}
\end{eqnarray*}
\end{lemma}

\begin{proof}
The estimate can be obtained by interpolating the four inequalities: 
\begin{equation*}
\left\Vert P_{N}^{x}P_{M}^{\xi }\tilde{Q}^{+}\left(
P_{N_{1}}^{x}P_{M_{1}}^{\xi }\tilde{f},P_{N_{2}}^{x}P_{M_{2}}^{\xi }\tilde{g}%
\right) \right\Vert _{L_{\xi }^{r}}\lesssim \left\Vert
P_{N_{1}}^{x}P_{M_{1}}^{\xi }\tilde{f}\right\Vert _{L_{\xi }^{p}}\left\Vert
P_{N_{2}}^{x}P_{M_{2}}^{\xi }\tilde{g}\right\Vert _{L_{\xi }^{q}}
\end{equation*}%
\begin{equation*}
\left\Vert P_{N}^{x}P_{M}^{\xi }\tilde{Q}^{+}\left(
P_{N_{1}}^{x}P_{M_{1}}^{\xi }\tilde{f},P_{N_{2}}^{x}P_{M_{2}}^{\xi }\tilde{g}%
\right) \right\Vert _{L_{\xi }^{r}}\lesssim \frac{\max (M_{1},M_{2})}{M}%
\left\Vert P_{N_{1}}^{x}P_{M_{1}}^{\xi }\tilde{f}\right\Vert _{L_{\xi
}^{p}}\left\Vert P_{N_{2}}^{x}P_{M_{2}}^{\xi }\tilde{g}\right\Vert _{L_{\xi
}^{q}}
\end{equation*}%
\begin{equation*}
\left\Vert P_{N}^{x}P_{M}^{\xi }\tilde{Q}^{+}\left(
P_{N_{1}}^{x}P_{M_{1}}^{\xi }\tilde{f},P_{N_{2}}^{x}P_{M_{2}}^{\xi }\tilde{g}%
\right) \right\Vert _{L_{\xi }^{r}}\lesssim \frac{\max (N_{1},N_{2})}{N}%
\left\Vert P_{N_{1}}^{x}P_{M_{1}}^{\xi }\tilde{f}\right\Vert _{L_{\xi
}^{p}}\left\Vert P_{N_{2}}^{x}P_{M_{2}}^{\xi }\tilde{g}\right\Vert _{L_{\xi
}^{q}}
\end{equation*}%
\begin{eqnarray*}
\left\Vert P_{N}^{x}P_{M}^{\xi }\tilde{Q}^{+}\left(
P_{N_{1}}^{x}P_{M_{1}}^{\xi }\tilde{f},P_{N_{2}}^{x}P_{M_{2}}^{\xi }\tilde{g}%
\right) \right\Vert _{L_{\xi }^{r}} &\lesssim &\frac{\max (N_{1},N_{2})}{N}%
\frac{\max (M_{1},M_{2})}{M} \\
&&\times \left\Vert P_{N_{1}}^{x}P_{M_{1}}^{\xi }\tilde{f}\right\Vert
_{L_{\xi }^{p}}\left\Vert P_{N_{2}}^{x}P_{M_{2}}^{\xi }\tilde{g}\right\Vert
_{L_{\xi }^{q}}
\end{eqnarray*}

The main concern is definitely the $\xi $-part as $\tilde{Q}^{+}$ commutes
with $P_{N}^{x}$ and $\nabla _{x},$ it suffices to prove 
\begin{equation}
\left\Vert P_{M}^{\xi }\tilde{Q}^{+}\left( P_{M_{1}}^{\xi }\tilde{f}%
,P_{M_{2}}^{\xi }\tilde{g}\right) \right\Vert _{L_{\xi }^{r}}\lesssim
\left\Vert P_{M_{1}}^{\xi }\tilde{f}\right\Vert _{L_{\xi }^{p}}\left\Vert
P_{M_{2}}^{\xi }\tilde{g}\right\Vert _{L_{\xi }^{q}},
\label{eqn:gain term freq 1}
\end{equation}

and 
\begin{equation}
\left\Vert P_{M}^{\xi }\tilde{Q}^{+}\left( P_{M_{1}}^{\xi }\tilde{f}%
,P_{M_{2}}^{\xi }\tilde{g}\right) \right\Vert _{L_{\xi }^{r}}\lesssim \frac{%
\max (M_{1},M_{2})}{M}\left\Vert P_{M_{1}}^{\xi }\tilde{f}\right\Vert
_{L_{\xi }^{p}}\left\Vert P_{M_{2}}^{\xi }\tilde{g}\right\Vert _{L_{\xi
}^{q}}.  \label{eqn:gain term freq 2}
\end{equation}%
While (\ref{eqn:gain term freq 1}) is directly from Lemma \ref%
{Lemma:HolderLikeInequalityForGain}, we prove only (\ref{eqn:gain term freq
2}). Notice that 
\begin{equation*}
\partial _{\xi _{j}}\left( \tilde{f}(\xi ^{+})\tilde{g}(\xi ^{-})\right)
=\left( \partial _{\xi _{j}}\tilde{f}\right) (\xi ^{+})\tilde{g}(\xi
^{-})\partial _{\xi _{j}}\xi ^{+}+\tilde{f}(\xi ^{+})\left( \partial _{\xi
_{j}}\tilde{g}\right) (\xi ^{-})\partial _{\xi _{j}}\xi ^{-}
\end{equation*}%
and $\left\vert \partial _{\xi _{j}}\xi ^{\pm }\right\vert \leqslant 1$, we
can thus apply Lemma \ref{Lemma:HolderLikeInequalityForGain} with $%
b=\partial _{\xi _{j}}\xi ^{\pm }$. That is, 
\begin{eqnarray*}
&&\left\Vert P_{M}^{\xi }\tilde{Q}^{+}\left( P_{M_{1}}^{\xi }\tilde{f}%
,P_{M_{2}}^{\xi }\tilde{g}\right) \right\Vert _{L_{\xi }^{r}} \\
&=&M^{-1}\left\Vert P_{M}^{\xi }\nabla _{\xi }\tilde{Q}^{+}\left(
P_{M_{1}}^{\xi }\tilde{f},P_{M_{2}}^{\xi }\tilde{g}\right) \right\Vert
_{L_{\xi }^{r}} \\
&\lesssim &M^{-1}\left( \left\Vert P_{M_{1}}^{\xi }\nabla _{\xi }\tilde{f}%
\right\Vert _{L_{\xi }^{p}}\left\Vert P_{M_{2}}^{\xi }\tilde{g}\right\Vert
_{L_{\xi }^{q}}+\left\Vert P_{M_{1}}^{\xi }\tilde{f}\right\Vert _{L_{\xi
}^{p}}\left\Vert P_{M_{2}}^{\xi }\nabla _{\xi }\tilde{g}\right\Vert _{L_{\xi
}^{q}}\right)
\end{eqnarray*}%
which, by Bernstein, is 
\begin{equation*}
\left\Vert P_{M}^{\xi }\tilde{Q}^{+}\left( P_{M_{1}}^{\xi }\tilde{f}%
,P_{M_{2}}^{\xi }\tilde{g}\right) \right\Vert _{L_{\xi }^{r}}\lesssim \frac{%
\max (M_{1},M_{2})}{M}\left\Vert P_{M_{1}}^{\xi }\tilde{f}\right\Vert
_{L_{\xi }^{p}}\left\Vert P_{M_{2}}^{\xi }\tilde{g}\right\Vert _{L_{\xi
}^{q}}
\end{equation*}%
as claimed in (\ref{eqn:gain term freq 2}).
\end{proof}

\subsection{Proof of Theorem \protect\ref{T:gain}}

It suffice to prove 
\begin{equation*}
\Vert \left\langle \nabla _{x}\right\rangle ^{s}\left\langle \nabla _{\xi
}\right\rangle ^{s}\tilde{Q}^{+}(\tilde{f},\tilde{g})\Vert
_{L_{t}^{1}L_{x\xi }^{2}}\lesssim \Vert \tilde{f}\Vert _{X_{s,\frac{1}{2}%
+}}\Vert \tilde{g}\Vert _{X_{s,\frac{1}{2}+}}\text{.}
\end{equation*}%
We start by decomposing $\tilde{Q}^{+}(f,g)$ into Littlewood-Paley pieces 
\begin{eqnarray*}
&&\Vert \left\langle \nabla _{x}\right\rangle ^{s}\left\langle \nabla _{\xi
}\right\rangle ^{s}\tilde{Q}^{+}(\tilde{f},\tilde{g})\Vert
_{L_{t}^{1}L_{x\xi }^{2}} \\
&\leqslant &\sum_{\substack{ M,M_{1},M_{2}  \\ N,N_{1},N_{2}}}%
M^{s}N^{s}\Vert P_{N}^{x}P_{M}^{\xi }\tilde{Q}^{+}\left(
P_{N_{1}}^{x}P_{M_{1}}^{\xi }\tilde{f},P_{N_{2}}^{x}P_{M_{2}}^{\xi }\tilde{g}%
\right) \Vert _{L_{t}^{1}L_{x\xi }^{2}} \\
&=&A+B+C+D
\end{eqnarray*}%
where we have seperated the sum into 4 cases in which case A is $%
M_{1}\geqslant M_{2}$, $N_{1}\geqslant N_{2}$, case $B$ is $M_{1}\leqslant
M_{2}$, $N_{1}\geqslant N_{2}$, case $C$ is $M_{1}\leqslant M_{2}$, $%
N_{1}\leqslant N_{2}$, and case $D$ is $M_{1}\geqslant M_{2}$, $%
N_{1}\leqslant N_{2}$. We handle only cases A and B as the other two cases
follow similarly.

\subsubsection{Case A: $M_{1}\geqslant M_{2}$, $N_{1}\geqslant N_{2}$}

\begin{equation*}
A\lesssim \sum_{\substack{ _{\substack{ M,M_{1},M_{2}  \\ N,N_{1},N_{2}}} 
\\ M_{1}\geqslant M_{2},N_{1}\geqslant N_{2}}}\frac{M^{s}N^{s}}{%
M_{1}^{s}N_{1}^{s}}\Vert P_{N}^{x}P_{M}^{\xi }\tilde{Q}^{+}\left(
P_{N_{1}}^{x}P_{M_{1}}^{\xi }\left\langle \nabla _{x}\right\rangle
^{s}\left\langle \nabla _{\xi }\right\rangle ^{s}\tilde{f}%
,P_{N_{2}}^{x}P_{M_{2}}^{\xi }\tilde{g}\right) \Vert _{L_{t}^{1}L_{x\xi
}^{2}}
\end{equation*}%
Use Lemma \ref{Lemma:FreqHolderForGain} and H\"{o}lder, 
\begin{eqnarray*}
A &\lesssim &\sum_{\substack{ _{\substack{ M,M_{1},M_{2}  \\ N,N_{1},N_{2}}} 
\\ M_{1}\geqslant M_{2},N_{1}\geqslant N_{2}}}\frac{M^{s}N^{s}}{%
M_{1}^{s}N_{1}^{s}}\min (1,\frac{N_{1}}{N})\min (1,\frac{M_{1}}{M}) \\
&&\times \Vert P_{N_{1}}^{x}P_{M_{1}}^{\xi }\left\langle \nabla
_{x}\right\rangle ^{s}\left\langle \nabla _{\xi }\right\rangle ^{s}\tilde{f}%
\Vert _{L_{t}^{2}L_{x\xi }^{3}}\Vert P_{N_{2}}^{x}P_{M_{2}}^{\xi }\tilde{g}%
\Vert _{L_{t}^{2}L_{x\xi }^{6}}.
\end{eqnarray*}%
By Sobolev and Bernstein, 
\begin{eqnarray*}
A &\lesssim &\sum_{\substack{ _{\substack{ M,M_{1},M_{2}  \\ N,N_{1},N_{2}}} 
\\ M_{1}\geqslant M_{2},N_{1}\geqslant N_{2}}}\frac{M^{s}N^{s}}{%
M_{1}^{s}N_{1}^{s}}\min (1,\frac{N_{1}}{N})\min (1,\frac{M_{1}}{M}) \\
&&\times \Vert P_{N_{1}}^{x}P_{M_{1}}^{\xi }\left\langle \nabla
_{x}\right\rangle ^{s}\left\langle \nabla _{\xi }\right\rangle ^{s}\tilde{f}%
\Vert _{L_{t}^{2}L_{x\xi }^{3}}\Vert P_{N_{2}}^{x}P_{M_{2}}^{\xi
}\left\langle \nabla _{x}\right\rangle ^{\frac{1}{2}}\left\langle \nabla
_{\xi }\right\rangle ^{\frac{1}{2}}\tilde{g}\Vert _{L_{t}^{2}L_{x\xi }^{3}}
\\
&\lesssim &\sum_{\substack{ _{\substack{ M,M_{1},M_{2}  \\ N,N_{1},N_{2}}} 
\\ M_{1}\geqslant M_{2},N_{1}\geqslant N_{2}}}\frac{M^{s}N^{s}}{%
M_{1}^{s}N_{1}^{s}}\min (1,\frac{N_{1}}{N})\min (1,\frac{M_{1}}{M})\frac{1}{%
N_{2}^{s-\frac{1}{2}}}\frac{1}{M_{2}^{s-\frac{1}{2}}} \\
&&\times \Vert P_{N_{1}}^{x}P_{M_{1}}^{\xi }\left\langle \nabla
_{x}\right\rangle ^{s}\left\langle \nabla _{\xi }\right\rangle ^{s}\tilde{f}%
\Vert _{L_{t}^{2}L_{x\xi }^{3}}\Vert P_{N_{2}}^{x}P_{M_{2}}^{\xi
}\left\langle \nabla _{x}\right\rangle ^{s}\left\langle \nabla _{\xi
}\right\rangle ^{s}\tilde{g}\Vert _{L_{t}^{2}L_{x\xi }^{3}}
\end{eqnarray*}%
By Strichartz, 
\begin{eqnarray*}
A &\lesssim &\sum_{\substack{ _{\substack{ M,M_{1},M_{2}  \\ N,N_{1},N_{2}}} 
\\ M_{1}\geqslant M_{2},N_{1}\geqslant N_{2}}}\frac{M^{s}N^{s}}{%
M_{1}^{s}N_{1}^{s}}\min (1,\frac{N_{1}}{N})\min (1,\frac{M_{1}}{M})\frac{1}{%
N_{2}^{s-\frac{1}{2}}}\frac{1}{M_{2}^{s-\frac{1}{2}}} \\
&&\times \Vert P_{N_{1}}^{x}P_{M_{1}}^{\xi }\tilde{f}\Vert _{X_{s,\frac{1}{2}%
+}}\Vert P_{N_{2}}^{x}P_{M_{2}}^{\xi }\tilde{g}\Vert _{X_{s,\frac{1}{2}+}}%
\text{.}
\end{eqnarray*}%
Cauchy-Schwarz in $N_{2}$,$M_{2}$ to carry out the $N_{2}$,$M_{2}$ sum, we
have 
\begin{equation*}
A\lesssim \Vert \tilde{g}\Vert _{X_{s,\frac{1}{2}+}}\sum_{_{\substack{ %
M,M_{1}  \\ N,N_{1}}}}\frac{M^{s}N^{s}}{M_{1}^{s}N_{1}^{s}}\min (1,\frac{%
N_{1}}{N})\min (1,\frac{M_{1}}{M})\Vert P_{N_{1}}^{x}P_{M_{1}}^{\xi }f\Vert
_{X_{s,\frac{1}{2}+}}
\end{equation*}%
Separate the sum into four cases $M_{1}\leqslant M,M_{1}\geqslant
M,N_{1}\leqslant N$, and $N_{1}\geqslant N$, we are done by Schur's test or
Cauchy-Schwarz.

\subsubsection{Case B: $M_{1}\leqslant M_{2}$, $N_{1}\geqslant N_{2}$}

\begin{equation*}
B\lesssim \sum_{\substack{ _{\substack{ M,M_{1},M_{2}  \\ N,N_{1},N_{2}}} 
\\ M_{1}\leqslant M_{2},N_{1}\geqslant N_{2}}}\frac{M^{s}N^{s}}{%
M_{2}^{s}N_{1}^{s}}\Vert P_{N}^{x}P_{M}^{\xi }\tilde{Q}^{+}\left(
P_{N_{1}}^{x}P_{M_{1}}^{\xi }\left\langle \nabla _{x}\right\rangle ^{s}%
\tilde{f},P_{N_{2}}^{x}P_{M_{2}}^{\xi }\left\langle \nabla _{\xi
}\right\rangle ^{s}\tilde{g}\right) \Vert _{L_{t}^{1}L_{x\xi }^{2}}
\end{equation*}%
Use Lemma \ref{Lemma:FreqHolderForGain} and H\"{o}lder, 
\begin{eqnarray*}
B &\lesssim &\sum_{\substack{ _{\substack{ M,M_{1},M_{2}  \\ N,N_{1},N_{2}}} 
\\ M_{1}\leqslant M_{2},N_{1}\geqslant N_{2}}}\frac{M^{s}N^{s}}{%
M_{2}^{s}N_{1}^{s}}\min (1,\frac{N_{1}}{N})\min (1,\frac{M_{2}}{M}) \\
&&\times \Vert P_{N_{1}}^{x}P_{M_{1}}^{\xi }\left\langle \nabla
_{x}\right\rangle ^{s}\tilde{f}\Vert _{L_{t}^{2}L_{x}^{3}L_{\xi }^{6}}\Vert
P_{N_{2}}^{x}P_{M_{2}}^{\xi }\left\langle \nabla _{\xi }\right\rangle ^{s}%
\tilde{g}\Vert _{L_{t}^{2}L_{x}^{6}L_{\xi }^{3}}.
\end{eqnarray*}%
By Sobolev and Bernstein, 
\begin{eqnarray*}
B &\lesssim &\sum_{\substack{ _{\substack{ M,M_{1},M_{2}  \\ N,N_{1},N_{2}}} 
\\ M_{1}\leqslant M_{2},N_{1}\geqslant N_{2}}}\frac{M^{s}N^{s}}{%
M_{2}^{s}N_{1}^{s}}\min (1,\frac{N_{1}}{N})\min (1,\frac{M_{2}}{M}) \\
&&\times \Vert P_{N_{1}}^{x}P_{M_{1}}^{\xi }\left\langle \nabla
_{x}\right\rangle ^{s}\left\langle \nabla _{\xi }\right\rangle ^{\frac{1}{2}}%
\tilde{f}\Vert _{L_{t}^{2}L_{x\xi }^{3}}\Vert P_{N_{2}}^{x}P_{M_{2}}^{\xi
}\left\langle \nabla _{x}\right\rangle ^{\frac{1}{2}}\left\langle \nabla
_{\xi }\right\rangle ^{s}\tilde{g}\Vert _{L_{t}^{2}L_{x\xi }^{3}} \\
&\lesssim &\sum_{\substack{ _{\substack{ M,M_{1},M_{2}  \\ N,N_{1},N_{2}}} 
\\ M_{1}\leqslant M_{2},N_{1}\geqslant N_{2}}}\frac{M^{s}N^{s}}{%
M_{2}^{s}N_{1}^{s}}\min (1,\frac{N_{1}}{N})\min (1,\frac{M_{2}}{M})\frac{1}{%
N_{2}^{s-\frac{1}{2}}}\frac{1}{M_{1}^{s-\frac{1}{2}}} \\
&&\times \Vert P_{N_{1}}^{x}P_{M_{1}}^{\xi }\left\langle \nabla
_{x}\right\rangle ^{s}\left\langle \nabla _{\xi }\right\rangle ^{s}\tilde{f}%
\Vert _{L_{t}^{2}L_{x\xi }^{3}}\Vert P_{N_{2}}^{x}P_{M_{2}}^{\xi
}\left\langle \nabla _{x}\right\rangle ^{s}\left\langle \nabla _{\xi
}\right\rangle ^{s}\tilde{g}\Vert _{L_{t}^{2}L_{x\xi }^{3}}
\end{eqnarray*}%
One can then proceed in the exact same manner as in Case A. We omit the
details.

\section{Well-posedness\label{S:well-posedness}}

For completeness, we provide a quick (and not so complete) treatment of the
well-posedness statement in Theorem \ref{T:well-ill-posedness} using (\ref%
{E:loss-2}) and (\ref{E:gain-1}). As mentioned before, this part is already
proved in \cite{CDP19a,CDP19b}. We prove only the existence, uniqueness, and
the Lipschitz continuity of the solution map here. The nonnegativity of $f$
follows from the persistance of regularity \cite{CDP19b}.

Write $R=2\left\Vert \tilde{f}_{0}\right\Vert _{H_{x\xi }^{s}}$ and let $%
\theta (t)$ be a smooth function such that $\theta (t)=1$ for all $%
\left\vert t\right\vert \leq 1$ and $\theta (t)=0$ for $\left\vert
t\right\vert \geq 2$. Put $T=\delta \langle R\rangle ^{-2}$, where $\delta
>0 $ is an absolute constant (selected depending upon constants appearing in
the estimates such as \eqref{E:gl} that we apply). The estimates %
\eqref{E:loss-2} and \eqref{E:gain-1} admit the following minor adjustments
that account for restriction in time: 
\begin{equation}
\Vert \theta (\frac{t}{T})\tilde{Q}(\tilde{f},\tilde{g})\Vert _{X_{s,-\frac{1%
}{2}+}}\lesssim T^{1/2}\Vert \tilde{f}\Vert _{X_{s,\frac{1}{2}+}}\Vert 
\tilde{g}\Vert _{X_{s,\frac{1}{2}+}}\text{ for }s>1.  \label{E:gl}
\end{equation}%
Denote $B=\left\{ \tilde{f}:\Vert \tilde{f}\Vert _{X_{s,\frac{1}{2}+}}\leq
R\right\} $ the ball of radius $R$ in $X_{s,\frac{1}{2}+}$ and introduce the
nonlinear mapping $\Lambda :B\rightarrow \mathcal{S}^{\prime }$ given by 
\begin{equation*}
\Lambda (\tilde{f})=\theta (t)e^{it\nabla _{x}\cdot \nabla _{\xi }}\tilde{f}%
_{0}+D(\tilde{f},\tilde{f})\,,\qquad D(\tilde{f},\tilde{f})=\theta
(t)\int_{0}^{t}e^{i(t-t^{\prime })\nabla _{x}\cdot \nabla _{\xi }}\theta
(t^{\prime }/T)\tilde{Q}(\tilde{f}(t^{\prime }),\tilde{f}(t^{\prime
}))\,dt^{\prime }.
\end{equation*}
A fixed point $\tilde{f}=\Lambda (\tilde{f})$ will satisfy (\ref{E:boltz2})/%
\eqref{E:boltz1} for $-T\leq t\leq T$. Using \eqref{E:gl} one can show that 
\begin{equation*}
\Lambda :B\rightarrow B\,,\qquad \Vert \Lambda (\tilde{f})-\Lambda (\tilde{g}%
)\Vert _{X_{s,\frac{1}{2}+}}\leq \tfrac{1}{2}\Vert \tilde{f}-\tilde{g}\Vert
_{X_{s,\frac{1}{2}+}}
\end{equation*}%
and thus $\Lambda $ is a contraction. From this it follows that there is a
fixed point $\tilde{f}$ in $X_{s,\frac{1}{2}+}$, and this fixed point is
unique in $X_{s,\frac{1}{2}+}$.

Lipschitz continuity of the data to solution map is also straightforward.
Given two initial conditions $\tilde{f}_{0}$ and $\tilde{g}_{0}$, let $%
R=2\max (\Vert \tilde{f}_{0}\Vert _{H^{s}},\Vert \tilde{g}_{0}\Vert
_{H^{s}}) $ and $\tilde{f}$, $\tilde{g}$ be the corresponding fixed points
of $\Lambda $ constructed above. Then 
\begin{equation*}
\tilde{f}-\tilde{g}=\theta (t)e^{it\nabla _{x}\cdot \nabla _{\xi }}(\tilde{f}%
_{0}-\tilde{g}_{0})+D(\tilde{f}-\tilde{g},\tilde{f}-\tilde{g})+D(\tilde{f}-%
\tilde{g},\tilde{g})+D(\tilde{g},\tilde{f}-\tilde{g})
\end{equation*}%
The embedding $X_{s,\frac{1}{2}+}\hookrightarrow C\left( [-T,T],H_{x\xi
}^{s}\right) $ and estimate \eqref{E:gl} applied to the above equation give
the Lipschitz continuity of the data-to-solution map on the time $[-T,T]$.

\section{Ill-posedness\label{S:ill-posedness}}

In this section, we prove the ill-posedness portion of Theorem \ref%
{T:well-ill-posedness}. We will take 
\begin{equation*}
0<N\ll \max (M_{1},M_{2})^{-1}\ll 1\,,\qquad M_{1}\geq 1\,,\qquad M_{2}\geq 1
\end{equation*}

\subsection{Imprecise motivating calculations\label{s:motivate}}

We start by taking the two linear solutions that generated the
\textquotedblleft bad\textquotedblright\ interaction in the loss term from 
\S \ref{S:loss}, now scaled so that $\Vert \tilde{\phi}\Vert _{H_{x\xi
}^{s}}\sim 1$ and $\Vert \tilde{\psi}\Vert _{H_{x\xi }^{s}}\sim 1$. The bad
solution will have $O(1)$ $H_{x\xi }^{s}$-norm initially but more obvious
ill-behavior in $H_{x\xi }^{s_{0}}$-norm with $s_{0}<s$. The functions are: 
\begin{equation*}
\hat{\phi}(\eta _{1},v)=\frac{1}{M_{1}^{\frac{3}{2}+s}N^{3/2}}\hat{\chi}(%
\frac{\eta _{1}}{M_{1}})\hat{\chi}(\frac{v}{N})
\end{equation*}%
\begin{equation*}
\hat{\psi}(\eta _{2},v_{2})=\frac{1}{M_{2}^{1+s}N_{2}^{1+s}}\sum_{j=1}^{J}%
\hat{\chi}\left( \frac{P_{e_{j}}^{\perp }\eta _{2}}{M_{2}}\right) \hat{\chi}%
(N_{2}P_{e_{j}}\eta _{2})\hat{\chi}(M_{2}P_{e_{j}}^{\perp }v_{2})\hat{\chi}(%
\frac{P_{e_{j}}v_{2}}{N_{2}})
\end{equation*}%
where $J\sim M_{2}^{2}N_{2}^{2}$. Let $\tilde{f}=e^{it\nabla _{x}\cdot
\nabla _{\xi }}\tilde{\phi}$ and $\tilde{g}=e^{it\nabla _{x}\cdot \nabla
_{\xi }}\tilde{\psi}$. To deduce the form of the $(x,v)$ wave packets, we
need to pass from $\eta $ frequency support to $x$ spatial support. This is
done using the scaling principles of the Fourier transform. Since we have
chosen $0<N\ll \max (M_{1},M_{2})^{-1}\ll 1$, the $x$-support of $\tilde{f}$
hardly moves on an $O(1)$ time scale: 
\begin{equation*}
f(x,v,t)\approx \frac{M_{1}^{\frac{3}{2}-s}}{N^{3/2}}\chi (M_{1}x)\hat{\chi}(%
\frac{v}{N})
\end{equation*}%
On the other hand, the $g$ wave packet tubes move, each in their velocity
confined direction parallel to $e_{j}$ with speed $\sim N_{2}$. On short
time scales, we regard them as stationary: 
\begin{equation*}
g(x,v_{2},t)\approx \frac{M_{2}^{1-s}}{N_{2}^{2+s}}\sum_{j=1}^{J}\chi
(M_{2}P_{e_{j}}^{\perp }x)\chi (\frac{P_{e_{j}}x}{N_{2}})\hat{\chi}%
(M_{2}P_{e_{j}}^{\perp }v_{2})\hat{\chi}(\frac{P_{e_{j}}v_{2}}{N_{2}})
\end{equation*}

For the loss term, we need 
\begin{equation*}
\int_{\mathbb{R}^{3}}g(x,v_{2},t)\,dv_{2}\approx \frac{1}{%
M_{2}^{1+s}N_{2}^{1+s}}\sum_{j=1}^{J}\chi (M_{2}P_{e_{j}}^{\perp }x)\chi (%
\frac{P_{e_{j}}x}{N_{2}})\approx M_{2}^{1-s}N_{2}^{1-s}\chi (M_{2}x)
\end{equation*}%
We get 
\begin{equation*}
Q^{-}(f,g)\approx \frac{M_{1}^{\frac{3}{2}-s}}{N^{3/2}}%
M_{2}^{1-s}N_{2}^{1-s}\chi (M_{2}x)\chi (M_{1}x)\hat{\chi}(\frac{v}{N})
\end{equation*}%
Formal Duhamel iteration of $f+g$ suggests a prototype approximate solution
of the form 
\begin{equation*}
f_{a,0}(x,v,t)\approx \begin{aligned}[t]
&\frac{M_1^{\frac32-s}}{N^{3/2}}\exp[ -M_2^{1-s} N_2^{1-s} \chi( M_2 x)\; t
\,] \, \chi(M_1 x) \, \hat \chi(\frac{v}{N}) \\ &+
\frac{M_2^{1-s}}{N_2^{2+s}} \sum_{j=1}^J \chi\left( M_2 P_{e_j}^\perp x
\right) \chi( \frac{P_{e_j} x}{N_2} ) \hat \chi(M_2 P_{e_j}^\perp v ) \hat
\chi( \frac{P_{e_j}v}{N_2} ) \end{aligned}
\end{equation*}%
The prototype approximate solution $f_{a,0}$ is just $f+g$ above with $f$
preceeded by an exponentially decaying factor in time (which is suggested by
formal Duhamel iteration). Notice that when $s<1$, the exponential term
decays substantially on the short time scale $\sim M_{2}^{s-1}N_{2}^{s-1}\ll
1$. Thus we seek to show that this is an approximate solution on a time
interval $|t|\leq M_{2}^{s-1}N_{2}^{1-s}\ln M_{2}$, which is long enough for
the size of the first term with exponential coefficient to change
substantially.

Let us take

\begin{equation*}
M=M_{1}=M_{2}\gg 1\,,\qquad N=M^{-1}\ll 1
\end{equation*}%
We calculate, for any $0<s\leq 1$, $s_{0}\geq 0$, 
\begin{equation}
\Vert f_{a,0}\Vert _{L_{v}^{2,s_{0}}H_{x}^{s_{0}}}\sim M^{s_{0}-s}\exp
[-(MN_{2})^{1-s}t]\langle (MN_{2})^{1-s}t\rangle ^{s_{0}}+(MN_{2})^{s_{0}-s}
\label{E:ni01}
\end{equation}%
In the second term involving the sum in $j$ over $J\sim (MN_{2})^{2}$ terms,
the velocity supports are almost disjoint and the square of the sum is
approximately the sum of the squares. Thus, if we take 
\begin{equation}
0<s<1\,,\qquad s_{0}=s-\frac{\ln \ln M}{\ln M}\,,\qquad T_{\ast }=-\delta
(MN_{2})^{s-1}\ln M\leq t\leq 0  \label{E:rTdef}
\end{equation}%
then at the ends of this interval: 
\begin{equation*}
\Vert f_{a,0}(0)\Vert _{L_{v}^{2,s_{0}}H_{x}^{s_{0}}}\leq \frac{1}{\ln M}\ll
1\,,\qquad \Vert f_{a,0}(T_{\ast })\Vert _{L_{v}^{2,s_{0}}H_{x}^{s_{0}}}\sim
M^{\delta }\gg 1
\end{equation*}%
Note that, as $M\nearrow \infty $, $s_{0}\nearrow s$, and this approximate
solution, in $H_{x\xi }^{s_{0}}$, starts very small in $H_{x\xi }^{s_{0}}$
at time $0$, and rapidly inflates at time $T_{\ast }<0$ to large size in $%
H_{x\xi }^{s_{0}}$ backwards in time. By considering the same approximate
solution starting at $T_{\ast }<0$ and evolving forward to time $0$, we have
an approximate solution that starts large and deflates to a small size in a
very short period of time.

\subsection{Precise reformulation\label{sec:precise forumulation of bad}}

The calculations in \S \ref{s:motivate} were imprecise (with $\approx $
symbols) although they served only to motivate the form of the actual
approximate solution $f_{a}$ we are going to use. From now on, the error
estimates will be made more precise. The following is a better way to write
the form of $f_{\text{a}}$ leading to a more accurate approximate solution.
Start with 
\begin{equation}
f_{\text{b}}(t,x,v)=\frac{M_{2}^{1-s}}{N_{2}^{2+s}}\sum_{j=1}^{J}\chi \left(
M_{2}P_{e_{j}}^{\perp }(x-vt)\right) \chi (\frac{P_{e_{j}}(x-vt)}{N_{2}}%
)\chi (M_{2}P_{e_{j}}^{\perp }v)\chi (10\frac{P_{e_{j}}(v-N_{2}e_{j})}{N_{2}}%
)  \label{def:f_b}
\end{equation}%
We note that $f_{\text{b}}$ is a linear solution: 
\begin{equation*}
\partial _{t}f_{\text{b}}+v\cdot \nabla _{x}f_{\text{b}}=0
\end{equation*}%
Another feature is that if $|i-j|\geq 4$, then the $i$th and $j$th terms
have disjoint $v$-support. Now set 
\begin{equation}
f_{\text{r}}(t,x,v)=\frac{M_{1}^{\frac{3}{2}-s}}{N^{3/2}}\exp \left[
-\int_{s=0}^{t}\int_{v}f_{\text{b}}(s,x,v)\,dv\,ds\right] \chi (M_{1}x)\chi (%
\frac{v}{N})\text{, } \quad N\leqslant M^{-1}  \label{def:f_r}
\end{equation}%
Note that $f_{\text{r}}$ is a solution to a drift-free linearized Boltzmann
equation: 
\begin{equation*}
\partial _{t}f_{\text{r}}(t,x,v)=-f_{\text{r}}(t,x,v)\int_{v}f_{\text{b}%
}(t,x,v)\,dv=-Q^{-}(f_{\text{r}},f_{\text{b}})
\end{equation*}%
The approximate solution of interest is then 
\begin{equation*}
f_{\text{a}}=f_{\text{r}}+f_{\text{b}}
\end{equation*}%
Let 
\begin{equation}
\begin{aligned} F_{\text{err}}& =\partial _{t}f_{\text{a}}+v\cdot \nabla
_{x}f_{\text{a}}+Q^{-}(f_{\text{a}},f_{\text{a}})-Q^{+}(f_{\text{a}},f_{%
\text{a}}) \\ & =v\cdot \nabla
_{x}\,f_{\text{r}}+Q^{-}(f_{\text{b}},f_{\text{r}})+Q^{-}(f_{\text{r}},f_{%
\text{r}})\\ & \qquad +Q^{-}(f_{\text{b}},f_{\text{b}})
-Q^{+}(f_{\text{a}},f_{\text{a}}) \end{aligned}  \label{E:Ferr02}
\end{equation}%
Note that this is missing the \textquotedblleft bad terms\textquotedblright\ 
$v\cdot \nabla _{x}f_{\text{b}}$ and $Q^{-}(f_{\text{r}},f_{\text{b}})$
since $f_{\text{a}}$ already incorporates the evolutionary impact of these
interactions. Indeed $Q^{-}(f_{\text{r}},f_{\text{b}})$ is the cause of the
backwards-in-time growth of $f_{\text{r}}$.

For convenience, let us take 
\begin{equation*}
M=M_1=M_2\gg 1 \,, \qquad N_2^{1-s} \geq M^\delta \,, \qquad N\leq M^{-1}
\end{equation*}

We will show in Prop. \ref{P:norm_deflation} that there exist $f_\text{c}$
(for ``correction'') so that $f_{\text{ex}}$ (for ``exact'') given by 
\begin{equation*}
f_{\text{ex}} = f_{\text{a}}+f_{\text{c}}
\end{equation*}
exactly solves 
\begin{equation*}
\partial_t f_{\text{ex}} + v\cdot \nabla_x f_{\text{ex}} = -Q^-(f_{\text{ex}%
},f_{\text{ex}})+Q^+(f_{\text{ex}},f_{\text{ex}})
\end{equation*}
where 
\begin{equation*}
T_*=-\delta (MN_{2})^{s-1}\ln M\leq t\leq 0
\end{equation*}
and 
\begin{equation}  \label{E:fc_bound}
\|f_{\text{c}}(t) \|_{L_v^{2,1}H_x^1} \lesssim M^{-\delta/4}
\end{equation}
where the exponential term comes from the iteration of the local theory
approximation argument below over time steps of length $\delta (MN)^{s-1}$.
The equation for $f_{\text{c}}$ is 
\begin{equation}
\begin{aligned} &\partial_t f_{\text{c}} + v\cdot \nabla_x f_{\text{c}} = G
\\ &G= \pm Q^\pm(f_{\text{c}},f_{ \text{a}}) \pm
Q^\pm(f_{\text{a}},f_{\text{c}}) \pm Q^\pm(f_{\text{c}},f_{\text{c}
})-F_{\text{err}} \end{aligned}  \label{E:fc_equation}
\end{equation}%
To prove Proposition \ref{P:norm_deflation}, we will do estimates in local
time steps of size $\delta (MN_{2})^{s-1}\ll 1$ by working with the norm: 
\begin{equation*}
\Vert f_{\text{c}}(t)\Vert _{Z}=M\Vert f_{\text{c}}(t)\Vert
_{L_{v}^{2,1}L_{x}^{2}}+\Vert \nabla _{x}f_{\text{c}}(t)\Vert
_{L_{v}^{2,1}L_{x}^{2}}+\Vert f_{\text{c}}(t)\Vert _{L_{v}^{1}L_{x}^{\infty
}}+M^{-1}\Vert \nabla _{x}f_{\text{c}}(t)\Vert _{L_{v}^{1}L_{x}^{\infty }}
\end{equation*}%
to complete a perturbation argument to prove that $f_{\text{c}}$ (and thus $%
f_{\text{ex}}$) exists. As we are trying to construct a fairly smooth
solution with specific properties (some of which cannot be efficiently
treated using Strichartz and, in any case, we do not have (\ref{E:loss-2})
for $s=1$) to the Boltzmann equation linearized near a special solution, we
use the $Z$-norm instead of only the $L_{v}^{2,1}H_{x}^{1}$ norm to better
serve our purpose here. Inside the $Z$-norm, the $L_{v}^{1}L_{x}^{\infty }$
norm is preserved by the linear flow and scales the same way as the norm $%
\dot{L}_{v}^{2,\frac{3}{2}}\dot{H}_{x}^{3/2}$, and thus lies at regularity
scale \emph{above} $L_{v}^{2,1}H_{x}^{1}$. Nevertheless, due to the specific
structure of $f_{\text{a}}(t)$, we have (see Lemma \ref{L:fa_size} below) 
\begin{equation*}
\Vert f_{\text{a}}(t)\Vert _{Z}\sim \Vert f_{\text{a}}(t)\Vert
_{L_{v}^{2,1}H_{x}^{1}}
\end{equation*}%
and thus it is effective for our purposes. The extra terms $M\Vert f_{\text{c%
}}(t)\Vert _{L_{v}^{2,1}L_{x}^{2}}$ and $M^{-1}\Vert \nabla _{x}f_{\text{c}%
}(t)\Vert _{L_{v}^{1}L_{x}^{\infty }}$ are added to the definition of the
norm to account for the case in which $\nabla _{x}$ lands on the
\textquotedblleft wrong\textquotedblright\ term in the $Z$ bilinear estimate
(see Lemma \ref{L:Z_bilinear} below).

In the next three subsections, Lemmas \ref{L:fa_size}, \ref{L:Ferr_bound},
and \ref{L:Z_bilinear} are proved and will provide the key technical tools
to carry out the proof of Proposition \ref{P:norm_deflation}.

\subsection{Size of the approximate solution $f_{\text{a}}\label%
{sec:property of f_a}$}

Before proceeding, we will give a useful pointwise bound on $f_{\text{b}}$.
For $-\frac{1}{4}\leq t\leq \frac{1}{4}$, given the constraints on $v$, 
\begin{align*}
\hspace{0.3in}& \hspace{-0.3in}M^{1-s}N_{2}^{-2-s}\sum_{j=1}^{J}\chi
(10MP_{e_{j}}^{\perp }x)\chi (\frac{10P_{e_{j}}x}{N_{2}})\chi
(MP_{e_{j}}^{\perp }v)\chi (\frac{10P_{e_{j}}(v-N_{2}e_{j})}{N_{2}}) \\
& \leq f_{\text{b}}(t,x,v)\leq M^{1-s}N_{2}^{-2-s}\sum_{j=1}^{J}\chi (\frac{%
MP_{e_{j}}^{\perp }x}{10})\chi (\frac{P_{e_{j}}x}{10N_{2}})\chi
(MP_{e_{j}}^{\perp }v)\chi (\frac{P_{e_{j}}(v-N_{2}e_{j})}{10N_{2}})
\end{align*}%
Now upon integrating in $v$, we obtain 
\begin{align*}
\hspace{0.3in}& \hspace{-0.3in}M^{-1-s}N_{2}^{-1-s}\sum_{j=1}^{J}\chi
(10MP_{e_{j}}^{\perp }x)\chi (\frac{10P_{e_{j}}x}{N_{2}}) \\
& \leq \int f_{\text{b}}(t,x,v)\,dv\leq
M^{-1-s}N_{2}^{-1-s}\sum_{j=1}^{J}\chi (\frac{MP_{e_{j}}^{\perp }x}{10})\chi
(\frac{P_{e_{j}}x}{10N_{2}})
\end{align*}%
We can pointwise estimate this sum by%
\begin{equation}
\int f_{\text{b}}(t,x,v)\,dv\sim M^{-1-s}N_{2}^{-1-s}\left( \frac{N_{2}}{%
|x|+M^{-1}}\right) ^{2}\chi (\frac{x}{N_{2}}).
\label{e:estimate for integrate f_b-1}
\end{equation}%
To see (\ref{e:estimate for integrate f_b-1}), we split the size of $%
\left\vert x\right\vert $ into 3 cases. For $\left\vert x\right\vert
\leqslant M^{-1}$, all the tubes are overlapping and hence, the $J$ sum has $%
(MN_{2})^{2}$ summands. For $\left\vert x\right\vert \sim N_{2}$, the tubes
are effectively disjoint (up to tubes of distance comparable to $M^{-1}$),
that is, the $J$ sum has $O(1)$ summands. For the intermediate case, the
overlap count would be the number of points with adjacent distance $\frac{%
\left\vert x\right\vert }{MN_{2}}$ in a $M^{-1}\times M^{-1}$ box, and is
thus $M^{-2}/(\frac{\left\vert x\right\vert }{MN_{2}})^{2}$. Hence, we
conclude (\ref{e:estimate for integrate f_b-1}).

In particular, we have 
\begin{equation}
\int f_{\text{b}}(t,x,v)\,dv\sim M^{1-s}N_{2}^{1-s}\text{ for }\left\vert
x\right\vert \lesssim M^{-1}.  \label{e:estimate for integrate f_b-2}
\end{equation}%
The inequalities (\ref{e:estimate for integrate f_b-1}) and (\ref{e:estimate
for integrate f_b-2}) remain true with the extra factor $M^k$ if $\nabla_x^k$
is applied, for $k\geq 0$ and $k\in \mathbb{Z}$. Let 
\begin{equation*}
\beta(t,x) = \int_0^t \int f_{\text{b}}(t_0,x,v)\,dv \, dt_0
\end{equation*}
Then \eqref{e:estimate for integrate f_b-1} yields the pointwise upper
bounds 
\begin{equation*}
|\nabla_x^k \beta(t,x)| \lesssim |t| M^{-1+k-s}N_{2}^{-1-s}\left( \frac{N_{2}%
}{|x|+M^{-1}}\right) ^{2}\chi (\frac{x}{N_{2}})
\end{equation*}
which implies 
\begin{equation}  \label{E:beta-bound}
\|\beta(t,x)\|_{L_x^\infty} \lesssim |t| M^{1+k-s}N_2^{1-s}
\end{equation}

\begin{lemma}[bounds on $f_{\text{a}}$]
\label{L:fa_size} For all $q \geq 0$, we have for $t\leq 0$, 
\begin{equation}  \label{E:fa_size_q}
\| f_{\text{a}}(t) \|_{L_v^{2,q} H_x^q} \sim M^{q-s} \exp[|t| (MN_2)^{1-s}]
\langle |t| (MN_2)^{1-s} \rangle + (MN_2)^{q-s}
\end{equation}
In addition, 
\begin{equation*}
\| f_{\text{a}}(t) \|_Z \lesssim M^{1-s}\exp[|t| (MN_2)^{1-s}] \langle |t|
(MN_2)^{1-s} \rangle + (MN_2)^{1-s}
\end{equation*}
Since we further assume that $N_2^{1-s}\geq M^\delta$, this bound simplifies
to 
\begin{equation*}
\| f_{\text{a}}(t) \|_Z \lesssim (MN_2)^{1-s}
\end{equation*}
\end{lemma}

Before proceeding to the proof, note that if we take $q=s_{0}$ in %
\eqref{E:fa_size_q}, where 
\begin{equation}
s_{0}=s-\frac{\ln \ln M}{\ln M}  \label{E:r_def}
\end{equation}%
then 
\begin{equation}
\Vert f_{\text{a}}(t)\Vert _{L_{v}^{2,s_{0}}H_{x}^{s_{0}}}\sim \frac{1}{\ln M%
}\exp [|t|(MN_{2})^{1-s}] \langle |t| (MN_2)^{1-s} \rangle
\label{E:fa_size_r}
\end{equation}%
In particular, 
\begin{equation*}
\Vert f_{\text{a}}(0)\Vert _{L_{v}^{2,s_{0}}H_{x}^{s_{0}}}\sim \frac{1}{\ln M%
}\ll 1\,,\qquad \Vert f_{\text{a}}(T_{\ast })\Vert
_{L_{v}^{2,s_{0}}H_{x}^{s_{0}}}\sim \delta M^\delta \gg 1
\end{equation*}

Note that, as $M\nearrow \infty $, $s_0\nearrow s$, and this approximate
solution, in $L_v^{2,s_0}H_{x}^{s_0}$, starts very small in $%
L_v^{2,s_0}H_{x}^{s_0}$ at time $0$, and rapidly inflates at time $T_*<0$ to
large size in $L_v^{2,s_0}H_{x}^{s_0}$ backwards in time. By considering the
same approximate solution starting at $T_*<0$ and evolving forward to time $%
0 $, we have an approximate solution that starts large and deflates to a
small size in a very short period of time.

\begin{proof}
As $f_{b}$ and $f_{r}$'s velocity supports are disjoint, $\Vert f_{\text{a}%
}(t)\Vert _{X}\sim \Vert f_{\text{r}}(t)\Vert _{X}+\Vert f_{\text{b}%
}(t)\Vert _{X}$ where $X$ stands for the norms being considered. As we would
obtain the $\Vert f_{\text{r}}(t)\Vert _{X}$ and $\Vert f_{\text{b}}(t)\Vert
_{X}$ in Lemma \ref{L:fa_size} via suitable scaling of the computation in
the middle of the proof of Lemma \ref{L:Ferr_bound}, we are not repeating
the calculation here.
\end{proof}

\begin{remark}
\label{Remark: finite 2nd moment}One can compute directly from (\ref{def:f_b}%
) and (\ref{def:f_r}) that%
\begin{equation*}
\int \int (1+\left\vert x\right\vert ^{2}+\left\vert v\right\vert
^{2})f_{a}(t,x,v)dxdv\lesssim M^{-1-s}N_{2}^{4}+M^{-\frac{3}{2}-s+\delta }N^{%
\frac{3}{2}}
\end{equation*}%
As we can take $N_{2}\sim M^{\delta }$ and $N\leqslant M^{-1}$, the 2nd
moments of $f_{a}$ is uniformly bounded in $M$. That is, an uniform bound of
the 2nd moment will not turn the ill-posedness into well-posedness.
\end{remark}

\subsection{Bounding of forcing terms $F_{\text{err}}$}

\begin{lemma}[bounds on $F_{\text{err}}$]
\label{L:Ferr_bound} For $s>1/2$ and for $T_* \leq t \leq 0$, 
\begin{equation}
\Vert \int_{\tau }^{t}e^{-(t-t_{0})v\cdot \nabla _{x}}F_{\text{err}%
}(t_{0})\,dt_{0}\Vert _{Z}\lesssim M^{-\delta }  \label{E:Ferr01}
\end{equation}
\end{lemma}

\begin{proof}
We address all of the terms in \eqref{E:Ferr02} separately below. The terms $%
Q^{\pm }(f_{b},f_{b})$\emph{\ }need the most attention and require $s>\frac{1%
}{2}$, thus we put them last. For many terms, the estimate is achieved by
moving the $t_{0}$ integration to the outside: 
\begin{equation*}
\Vert \int_{\tau }^{t}e^{-(t-t_{0})v\cdot \nabla _{x}}F_{\text{err}%
}(t_{0})\,dt_{0}\Vert _{Z}\lesssim (MN_{2})^{s-1}\ln M\Vert F_{\text{err}%
}\Vert _{L_{t}^{\infty }Z}
\end{equation*}%
The only exception is the treatment of the bound on $L_{v}^{1}L_{x}^{\infty
} $ of $Q^{\pm }(f_{b},f_{b})$, where a substantial gain is captured by
carrying out the $t_{0}$ integration first.
\end{proof}

\subsubsection{Analysis of $v\cdot \protect\nabla _{x}f_{\text{r}}$}

Starting with 
\begin{equation*}
f_{\text{r}}(t,x,v)=\frac{M^{\frac{3}{2}-s}}{N^{3/2}}\exp \left[ -\beta (x,t)%
\right] \chi (M x)\chi (\frac{v}{N})
\end{equation*}%
and recalling the estimate \eqref{E:beta-bound} and that $|t|\leq \delta
(MN_2)^{s-1}\ln M$ 
\begin{align*}
\Vert v\cdot \nabla _{x}f_{\text{r}}\Vert _{L_{v}^{2,1}H_{x}^{1}}& \lesssim
NM^{\frac{3}{2}-s}\Big(M\Vert \exp [-\beta(t,x)](\nabla \chi )(M x)\Vert
_{L_{x}^{2}} \\
& \qquad \qquad \qquad +\Vert \nabla _{x}\beta(t,x)\exp [-\beta(t,x)]\chi (M
x)\Vert _{L_{x}^{2}}\Big) \\
& \lesssim NM^{\frac{3}{2}-s}\Big(M^{-1/2}\Vert \exp [-\beta (x,t)]\Vert
_{L_{x}^{\infty }} \\
& \qquad \qquad \qquad +M^{-3/2}\Vert \nabla _{x}\beta(t,x)\Vert
_{L_{x}^{\infty }}\Vert \exp [-\beta(t,x)]\Vert _{L_{x}^{\infty }}\Big) \\
& \lesssim NM^{1-s}\exp [\delta \ln M]\ln M=NM^{1+\delta -s}\langle \delta
\ln M \rangle \leq M^{\delta -s}\langle \delta \ln M \rangle
\end{align*}

When multiplied by $|T^*|$, this produces a bound $M^{2\delta -1} N_2^{s-1}$%
, which suffices provided $\delta < \frac13$. A similar bound is obtained
for $M^{-1}\Vert \nabla_x [ v\cdot \nabla _{x}f_{\text{r}}]\Vert
_{L_{v}^{2,1}H_{x}^{1}}$.

We also obtain the bound 
\begin{align*}
\hspace{0.3in}&\hspace{-0.3in} \Vert v\cdot \nabla _{x}f_{\text{r}}\Vert
_{L_{v}^{1}L_{x}^{\infty }}+M^{-1}\Vert \nabla _{x}[v\cdot \nabla _{x}f_{%
\text{r}}]\Vert _{L_{v}^{1}L_{x}^{\infty }} \\
&\lesssim M^{\frac52-s}N^{5/2} (1+ M^{-1}\|\nabla_x
\beta(t,x)\|_{L_x^\infty}) \|\exp[ - \beta(t,x)]\|_{L_x^\infty} \\
&\lesssim M^{-s} \langle \delta \ln M \rangle M^\delta
\end{align*}
Upon multiplying by $|T_*|$, we obtain the bound $M^{2\delta-1} N_2^{s-1}$,
which suffices provided $\delta<\frac13$.

\subsubsection{Analysis of $Q^{+}(f_{\text{r}},f_{\text{b}})$\label{SS:542}}

By \cite[Theorem 2]{AC10} with $\lambda =0$, $r=1$, $p=1$, $q=1$, 
\begin{equation*}
\Vert Q^{+}(f_{\text{r}},f_{\text{b}})\Vert _{L_{v}^{1}L_{x}^{\infty
}}\lesssim \Vert f_{\text{r}}\Vert _{L_{v}^{1}L_{x}^{\infty }}\Vert f_{\text{%
b}}\Vert _{L_{v}^{1}L_{x}^{\infty }}
\end{equation*}%
Since 
\begin{equation*}
\Vert f_{\text{r}}\Vert _{L_{v}^{1}L_{x}^{\infty }}\sim M^{\frac{3}{2}%
-s}N^{3/2}\exp [|t|(MN_{2})^{1-s}]\leq M^{\delta-s}\,,\qquad \Vert f_{\text{b%
}}\Vert _{L_{v}^{1}L_{x}^{\infty }}\lesssim (MN_{2})^{1-s}
\end{equation*}%
when $N\leq M^{-1}$, we obtain 
\begin{equation}
\Vert Q^{+}(f_{\text{r}},f_{\text{b}})\Vert _{L_{v}^{1}L_{x}^{\infty
}}\lesssim M^{\delta-s}(MN_2)^{s-1}  \label{E:Ferr03}
\end{equation}
Upon multiplying by $|T_*|$, this yields a bound $M^{\delta-s}\langle \delta
\ln M\rangle$, which suffices provided $s>\frac12$ and $\delta<\frac14$.

On the other hand, 
\begin{equation*}
\Vert \nabla _{x}Q^{+}(f_{\text{r}},f_{\text{b}})\Vert _{L_{x}^{2}}\leq
Q^{+}(\Vert \nabla _{x}f_{\text{r}}\Vert _{L_{x}^{2}},\Vert f_{\text{b}%
}\Vert _{L_{x}^{\infty }})+Q^{+}(\Vert f_{\text{r}}\Vert _{L_{x}^{\infty
}},\Vert \nabla _{x}f_{\text{b}}\Vert _{L_{x}^{2}})
\end{equation*}%
A weight of $|v|$ is transferred to $f_{\text{b}}(u^{\ast })$, since $f_{%
\text{r}}(v^{\ast })$ forces $|v^{\ast }|\lesssim N$ which in turn implies $%
|P_{\omega }^{\perp }v|\leq N\ll 1$ via 
\begin{equation*}
v^{\ast }=P_{\omega }^{\perp }v+P_{\omega }^{\Vert }u
\end{equation*}%
Thus if $|v|\gg N$ (in particular if $|v|\geq 1$), then $|P_{\omega }^{\Vert
}v|\sim |v|$. Since 
\begin{equation*}
u^{\ast }=P_{\omega }^{\perp }u+P_{\omega }^{\Vert }v
\end{equation*}%
it follows that $|u^{\ast }|\geq |P_{\omega }^{\Vert }v|\sim |v|$. Now by 
\cite[Theorem 2]{AC10}, with $\lambda =0$, $r=2$, $p=1$, $q=2,$ 
\begin{equation*}
\Vert \nabla _{x}Q^{+}(f_{\text{r}},f_{\text{b}})\Vert
_{L_{v}^{2,1}L_{x}^{2}}\lesssim \Vert \nabla _{x}f_{\text{r}}\Vert
_{L_{v}^{1}L_{x}^{2}}\Vert f_{\text{b}}\Vert _{L_{v}^{2,1}L_{x}^{\infty
}}+\Vert f_{\text{r}}\Vert _{L_{v}^{1}L_{x}^{\infty }}\Vert \nabla _{x}f_{%
\text{b}}\Vert _{L_{v}^{2,1}L_{x}^{2}}
\end{equation*}%
(As an expository note, such a bound is not possible for $Q^{-}$. In that
case, the $L_{v}^{1}$ norm must go on $f_{\text{b}}$.) 
\begin{equation*}
\Vert \nabla _{x}f_{\text{r}}\Vert _{L_{v}^{1}L_{x}^{2}}\sim
M^{1-s}N^{3/2}\exp [(MN_{2})^{1-s}|t|]\,,\qquad \Vert f_{\text{b}}\Vert
_{L_{v}^{2,1}L_{x}^{\infty }}\sim M^{1-s}N_{2}^{\frac{1}{2}-s}
\end{equation*}%
\begin{equation*}
\Vert f_{\text{r}}\Vert _{L_{v}^{1}L_{x}^{\infty }}\sim M^{\frac{3}{2}%
-s}N^{3/2}\exp [(MN_{2})^{1-s}|t|]\,,\qquad \Vert \nabla _{x}f_{\text{b}%
}\Vert _{L_{v}^{2,1}L_{x}^{2}}\sim (MN_{2})^{1-s}
\end{equation*}%
Thus, with $N\leq M^{-1}$, 
\begin{equation}
\Vert \nabla _{x}Q^{+}(f_{\text{r}},f_{\text{b}})\Vert
_{L_{v}^{2,1}L_{x}^{2}}\lesssim (MN_{2})^{1-s}M^{-s}\exp [(MN_{2})^{1-s}|t|]
\label{E:Ferr04}
\end{equation}
Upon multiplying by $|T_*|$, we obtain a bound of $M^{\delta -s}$, which
suffices provided $\delta<\frac14$ and $s>\frac12$.

Beside \eqref{E:Ferr03}, \eqref{E:Ferr04}, the other two norms $M\Vert
\bullet \Vert _{L_{v}^{2,1}L_{x}^{2}}$ and $M^{-1}\Vert \nabla _{x}\bullet
\Vert _{L_{v}^{1}L_{x}^{\infty }}$ comprising $Z$ are similarly bounded
since the $x$-frequency of both terms in $Q^{+}(f_{\text{r}},f_{\text{b}})$
is $\sim M$.

\subsubsection{Analysis of $Q^{\pm }(f_{b},f_{r})$ and $Q^{\pm
}(f_{r},f_{r}) $}

These terms will be treated simultaneously with the two estimates 
\begin{equation}
\Vert \nabla _{x}Q^{\pm }(f_{1},f_{2})\Vert _{L_{v}^{2,1}L_{x}^{2}}+M\Vert
Q^{\pm }(f_{1},f_{2})\Vert _{L_{v}^{2,1}L_{x}^{2}}\lesssim M\Vert f_{1}\Vert
_{L_{v}^{2,1}L_{x}^{2}}\Vert f_{2}\Vert _{L_{v}^{1}L_{x}^{\infty }}
\label{E:Ferr06}
\end{equation}%
\begin{equation}
\Vert Q^{\pm }(f_{1},f_{2})\Vert _{L_{v}^{1}L_{x}^{\infty }}+M^{-1}\Vert
\nabla _{x}Q^{\pm }(f_{1},f_{2})\Vert _{L_{v}^{1}L_{x}^{\infty }}\lesssim
\Vert f_{1}\Vert _{L_{v}^{1}L_{x}^{\infty }}\Vert f_{2}\Vert
_{L_{v}^{1}L_{x}^{\infty }}  \label{E:Ferr07}
\end{equation}%
for 
\begin{equation*}
(1,2)\in \{(b,r),(r,r)\}
\end{equation*}%
We have 
\begin{equation*}
\Vert f_{r}\Vert _{L_{v}^{2,1}L_{x}^{2}}\sim M^{-s}\exp [(MN_{2})^{1-s}|t|]
\end{equation*}%
\begin{equation*}
\Vert f_{r}\Vert _{L_{v}^{1}L_{x}^{\infty }}\sim M^{\frac{3}{2}%
-s}N^{3/2}\exp [(MN_{2})^{1-s}|t|]\leq M^{-s}\exp [(MN_{2})^{1-s}|t|]
\end{equation*}%
\begin{equation*}
\Vert f_{b}\Vert _{L_{v}^{2,1}L_{x}^{2}}\sim M^{-s}N_{2}^{1-s}
\end{equation*}%
\begin{equation*}
\Vert f_{b}\Vert _{L_{v}^{1}L_{x}^{\infty }}\sim (MN_{2})^{1-s}
\end{equation*}%
Substituting $(f_{1},f_{2})=(f_{b},f_{r})$ and $(f_{1},f_{2})=(f_{r},f_{r})$
into the two bounds \eqref{E:Ferr06}, \eqref{E:Ferr07}, we obtain 
\begin{equation*}
\Vert Q^{\pm }(f_{b},f_{r})\Vert _{Z}\lesssim (MN_{2})^{1-s}M^{-s}\exp
[(MN_{2})^{1-s}|t|]
\end{equation*}%
\begin{equation*}
\Vert Q^{\pm }(f_{r},f_{r})\Vert _{Z}\lesssim M^{1-2s}\exp
[(MN_{2})^{1-s}|t|]
\end{equation*}%
which, upon multiplying by the time factor $T_{\ast }$, yield bounds of $%
M^{\delta-s}$ and $M^{\delta-s}N_2^{s-1}$, respectively, which suffice
provided $s>\frac12$ and $\delta<\frac14$.

\subsubsection{Analysis of $Q^{\pm }(f_{b},f_{b})\label{sec:Q+-(f_b,f_b)}$}

It suffices to estimate $\left\Vert Q^{\pm }(f_{b},f_{b})\right\Vert
_{L_{v}^{2,1}L_{x}^{2}}$ and 
\begin{equation*}
\left\Vert \int_{\tau }^{t}e^{-(t-t_{0})v\cdot \nabla _{x}}Q^{\pm
}(f_{b},f_{b})dt_{0}\right\Vert _{L_{v}^{1}L_{x}^{\infty }}
\end{equation*}%
without the derivatives due to the specific and clear structure of $f_{b}$.
These estimates rely on the geometric gain of some integrals based on the
fine structure of the nonlinear interactions. We start with the $%
L_{v}^{2,1}L_{x}^{2}$ estimates on which the smallness comes from the $%
L_{x}^{2}$ integral. The gain inside the $L_{v}^{1}L_{x}^{\infty }$ estimate
comes from the time integration in the Duhamel integral and uses extensively
the $X_{s,b}$ techniques.

\paragraph{The $L_{v}^{2,1}H_x^1$ estimates}

For either the gain or loss term, the weight on $v$ and the $x$-derivative
produce factors of $N_{2}$ and $M$ respectively, so we can reduce to
estimating the $L_{v}^{2}L_{x}^{2}$ norm. The loss term, the factor $%
\int_{v}f_{\text{b}}(x,v,t)\,dv$ effectively restricts to $|x|\lesssim
M^{-1} $, as in \eqref{e:estimate for integrate f_b-2}, and this effectively
truncates the $x$-tubes in the other $f_{\text{b}}(x,v,t)$ factor. 
\begin{equation*}
Q^{-}(f_{\text{b}},f_{\text{b}})\sim M^{2-2s}N_{2}^{-1-2s}\chi
(Mx)\sum_{j}\chi (MP_{e_{j}}^{\perp }v)\chi
(10N_{2}^{-1}P_{e_{j}}(v-N_{2}e_{j}))
\end{equation*}%
Upon applying the $L_{v}^{2}$ norm, we use the disjointness of $v$-supports
between terms in the sum to obtain 
\begin{equation*}
\Vert Q^{-}(f_{\text{b}},f_{\text{b}})\Vert _{L_{v}^{2}L_{x}^{2}}\lesssim
M^{2-2s}N_{2}^{-1-2s}M^{-3/2}(MN_{2})M^{-1}N_{2}^{1/2}=M^{\frac{1}{2}%
-2s}N_{2}^{\frac{1}{2}-2s}
\end{equation*}%
Thus 
\begin{equation*}
\Vert Q^{-}(f_{\text{b}},f_{\text{b}})\Vert _{L_{v}^{2,1}H_{x}^{1}}+M\Vert
Q^{-}(f_{\text{b}},f_{\text{b}})\Vert _{L_{v}^{2,1}L_{x}^{2}}\lesssim M^{%
\frac{3}{2}-2s}N_{2}^{\frac{3}{2}-2s}
\end{equation*}%
Upon multiplying by the time factor $|T_*|$, this yields a bound of $%
M^{\frac12-s}N_2^{\frac12-s}$, which suffices for $s>\frac12$ and $%
\delta\leq 2(s-\frac12)$.

The same gain in the effective $x$-width is achieved in the gain term $%
Q^{+}(f_{b},f_{b})$, although it is a little more subtle. Recall (\ref%
{def:f_b}) and write 
\begin{equation}
f_{\text{b}}(t,x,v)=\frac{M^{1-s}}{N_{2}^{2+s}}\sum_{j=1}^{J}I_{j}(t,x,v),
\label{E:fbsplit}
\end{equation}%
then 
\begin{equation*}
Q^{+}(f_{b},f_{b})=M^{2-2s}N_{2}^{-4-2s}\sum_{j,k}Q^{+}(I_{j},I_{k})
\end{equation*}%
and therefore 
\begin{equation*}
\Vert Q^{+}(f_{b},f_{b})\Vert _{L_{x}^{2}}\lesssim
M^{2-2s}N_{2}^{-4-2s}\left(
\sum_{j_{1},j_{2},k_{1},k_{2}}%
\int_{x}Q^{+}(I_{j_{1}},I_{k_{1}})Q^{+}(I_{j_{2}},I_{k_{2}})\,dx\right)
^{1/2}
\end{equation*}%
The $x$ supports of $I_{j_{1}}$, $I_{k_{1}}$, $I_{j_{2}}$, $I_{k_{2}}$ are
parallel to $e_{j_{1}}$, $e_{k_{1}}$, $e_{j_{2}}$, $e_{k_{2}}$,
respectively. Typically, at least two of these directions are transverse%
\footnote{%
The rare nearly parallel cases can be handled by a dyadic angular
decomposition and corresponding reduction in the summation count, similar to
that given below in the treatment of the $L_v^1L_x^\infty$ estimate for the
gain term.}, which means that the two $x$-tubes intersect to a set with
diameter $\sim M^{-1}$. Thus, upon carrying out the $x$ integration, we
obtain a factor $M^{-3}$: 
\begin{equation*}
\Vert Q^{+}(f_{b},f_{b})\Vert _{L_{x}^{2}}\lesssim M^{\frac{1}{2}%
-2s}N_{2}^{-4-2s}\left( \sum_{j_{1},j_{2},k_{1},k_{2}}Q^{+}(\hat{I}_{j_{1}},%
\hat{I}_{k_{1}})Q^{+}(\hat{I}_{j_{2}},\hat{I}_{k_{2}})\right) ^{1/2}
\end{equation*}%
where $\hat{I}$ is only the $v$-part.

Thus 
\begin{equation*}
\Vert Q^{+}(f_{b},f_{b})\Vert _{L_{v}^{2}L_{x}^{2}}\lesssim M^{\frac{1}{2}%
-2s}N_{2}^{-4-2s}\Vert Q^{+}(\hat{I},\hat{I})\Vert _{L_{v}^{2}}
\end{equation*}%
where $\hat{I}$ is the $x$-independent function 
\begin{equation*}
\hat{I}(v)=\sum_{j=1}^{J}\chi (M_{2}P_{e_{j}}^{\perp }v)\chi (10\frac{%
P_{e_{j}}(v-N_{2}e_{j})}{N_{2}})\sim \boldsymbol{1}_{\frac{9}{10}N_{2}\leq
|v|\leq \frac{11}{10}N_{2}}(v)
\end{equation*}%
Thus, using \cite[Theorem 2]{AC10} with $\lambda =0$, $r=2$, $p=1$, $q=2$
like before, 
\begin{equation*}
\Vert Q^{+}(\hat{I},\hat{I})\Vert _{L_{v}^{2}}\lesssim N_{2}^{9/2},
\end{equation*}
from which it follows that 
\begin{equation*}
\Vert Q^{+}(f_{b},f_{b})\Vert _{L_{v}^{2}L_{x}^{2}}\lesssim M^{\frac{1}{2}%
-2s}N_{2}^{\frac{1}{2}-2s}
\end{equation*}%
as in the loss case.

\paragraph{The $L_{v}^{1}L_{x}^{\infty }$ estimates}

We 1st deal with $Q^{-}(f_{b},f_{b})$ as it is shorter and prepares for the
gain term which is more difficult. We estimate the Duhamel term as that is
how it is going to be used%
\begin{equation*}
D^{-}=\int_{\tau }^{t}e^{-(t-t_{0})v\cdot \nabla
_{x}}Q^{-}(f_{b},f_{b})(t_{0})dt_{0}
\end{equation*}%
Plugging in (\ref{def:f_b}) which is a linear solution and (\ref{e:estimate
for integrate f_b-1}), we have%
\begin{eqnarray*}
D^{-} &\sim &\frac{M^{1-s}}{N_{2}^{2+s}}\sum_{j=1}^{J}\chi \left(
MP_{e_{j}}^{\perp }(x-vt)\right) \chi (\frac{P_{e_{j}}(x-vt)}{N_{2}})\chi
(MP_{e_{j}}^{\perp }v)\chi (10\frac{P_{e_{j}}(v-N_{2}e_{j})}{N_{2}}) \\
&&\times \int_{\tau }^{t}\frac{M^{-1-s}}{N_{2}^{1+s}}\left( \frac{N_{2}}{%
|x-v(t-t_{0})|+M^{-1}}\right) ^{2}\chi (\frac{x-v(t-t_{0})}{N_{2}})dt_{0}
\end{eqnarray*}%
The point is we pick up a $\frac{1}{\left\vert v\right\vert }$ in the $%
dt_{0} $ integral%
\begin{equation*}
\int_{\tau }^{t}\left( \frac{1}{|x-v(t-t_{0})|+M^{-1}}\right) ^{2}\chi (%
\frac{x-v(t-t_{0})}{N_{2}})dt_{0}
\end{equation*}%
which is like $N_{2}^{-1}$ due to the cutoff $\chi (10\frac{%
P_{e_{j}}(v-N_{2}e_{j})}{N_{2}})$, and a $M$ from actually carrying out the
1D $dt_{0}$ integral (with some leftovers having no effects under $%
L_{x}^{\infty }$). So we have 
\begin{equation*}
\left\Vert D^{-}\right\Vert _{L_{v}^{1}L_{x}^{\infty }}\lesssim \frac{M^{1-s}%
}{N_{2}^{2+s}}\frac{M^{-1-s}}{N_{2}^{-1+s}}\left( MN_{2}^{-1}\right) \left(
M^{-2}N_{2}\right) \left( MN_{2}\right) ^{2}\lesssim M^{1-2s}N_{2}^{1-2s}
\end{equation*}%
where the factors $M^{-2}N_{2}$ and $\left( MN_{2}\right) ^{2}$ come from $%
L_{v}^{1}$ and the number of summands in $J$ respectively. It is small if $s>%
\frac{1}{2}.$

We now turn to $Q^{+}(f_{b},f_{b})$. Just like the loss term, we would like
to exploit the time integration in the Duhamel%
\begin{equation*}
D^{+}=\int_{\tau }^{t}e^{-(t-t_{0})v\cdot \nabla
_{x}}Q^{+}(f_{b},f_{b})(t_{0})dt_{0}.
\end{equation*}%
Use the short hand (\ref{E:fbsplit})

\begin{eqnarray*}
D^{+} &=&\frac{M^{2-2s}}{N_{2}^{4+2s}}\sum_{k}^{J}\sum_{j}^{J}\int du\int_{%
\mathbb{S}^{2}}d\omega \int_{\tau }^{t}e^{-(t-t_{0})v\cdot \nabla
_{x}}I_{j}(t_{0},x,u^{\ast })I_{k}(t_{0},x,v^{\ast })dt_{0} \\
&:&=\frac{M^{2-2s}}{N_{2}^{4+2s}}\sum_{k}^{J}\sum_{j}^{J}\int du\int_{%
\mathbb{S}^{2}}d\omega S_{j,k}(t,x,u^{\ast },v^{\ast }).
\end{eqnarray*}%
We estimate by%
\begin{eqnarray*}
\left\Vert D^{+}\right\Vert _{L_{v}^{1}L_{x}^{\infty }} &\leqslant &\frac{%
M^{2-2s}}{N_{2}^{4+2s}}\left\Vert \sum_{k}^{J}\sum_{j}^{J}\int du\int_{%
\mathbb{S}^{2}}d\omega S_{j,k}(t,x,u^{\ast },v^{\ast })\right\Vert
_{L_{v}^{1}L_{x}^{\infty }} \\
&\leqslant &\frac{M^{2-2s}}{N_{2}^{4+2s}}\sum_{k}^{J}\sum_{j}^{J}\int_{%
\mathbb{S}^{2}}d\omega \int du\int dv\left\Vert S_{j,k}(t,x,u^{\ast
},v^{\ast })\right\Vert _{L_{x}^{\infty }} \\
&=&\frac{M^{2-2s}}{N_{2}^{4+2s}}\sum_{k}^{J}\sum_{j}^{J}\int_{\mathbb{S}%
^{2}}d\omega \int du^{\ast }\int dv^{\ast }\left\Vert S_{j,k}(t,x,u^{\ast
},v^{\ast })\right\Vert _{L_{x}^{\infty }} \\
&=&\frac{M^{2-2s}}{N_{2}^{4+2s}}\sum_{k}^{J}\sum_{j}^{J}\int du^{\ast }\int
dv^{\ast }\int_{\mathbb{S}^{2}}d\omega \left\Vert S_{j,k}(t,x,u^{\ast
},v^{\ast })\right\Vert _{L_{x}^{\infty }}
\end{eqnarray*}%
We turn to 
\begin{equation*}
\int_{\mathbb{S}^{2}}d\omega \left\Vert S_{j,k}(t,x,u^{\ast },v^{\ast
})\right\Vert _{L_{x}^{\infty }}.
\end{equation*}%
Notice that, $I_{j}(t_{0},x,u^{\ast })$ or $I_{k}(t_{0},x,v^{\ast })$ are
not linear solutions in $\left( t,x,v\right) $ but in $\left( t,x,u^{\ast
}\right) $ or $\left( t,x,v^{\ast }\right) $, doing $dt_{0}$ 1st does not
net us a $1/\left\vert v\right\vert $ directly like in the loss term. Recall%
\begin{eqnarray*}
v^{\ast } &=&P_{\omega }^{\Vert }u+P_{\omega }^{\bot }v,\text{ }u^{\ast
}=P_{\omega }^{\Vert }v+P_{\omega }^{\bot }u, \\
v &=&P_{\omega }^{\bot }v^{\ast }+P_{\omega }^{\Vert }u^{\ast },\text{ }%
u=P_{\omega }^{\Vert }v^{\ast }+P_{\omega }^{\bot }u^{\ast },
\end{eqnarray*}%
we have 
\begin{equation*}
v-u^{\ast }=P_{\omega }^{\bot }(v^{\ast }-u^{\ast }),v-v^{\ast }=-P_{\omega
}^{\Vert }(v^{\ast }-u^{\ast }),
\end{equation*}%
and hence%
\begin{eqnarray*}
x-v(t-t_{0})-u^{\ast }t_{0} &=&x-vt+P_{\omega }^{\bot }(v^{\ast }-u^{\ast
})t_{0}, \\
x-v(t-t_{0})-v^{\ast }t_{0} &=&x-vt-P_{\omega }^{\Vert }(v^{\ast }-u^{\ast
})t_{0}.
\end{eqnarray*}

Hence, fixing $u^{\ast }$,$v^{\ast }$, we have%
\begin{eqnarray*}
&&S_{j,k}(t,x,u^{\ast },v^{\ast }) \\
&\lesssim &\int_{0}^{t}\chi \left( MP_{e_{j}}^{\perp }(x-vt+P_{\omega
}^{\bot }(v^{\ast }-u^{\ast })t_{0})\right) \chi (\frac{P_{e_{j}}(x-vt+P_{%
\omega }^{\bot }(v^{\ast }-u^{\ast })t_{0})}{N_{2}}) \\
&&\chi (MP_{e_{j}}^{\perp }v^{\ast })\chi (10\frac{P_{e_{j}}(v^{\ast
}-N_{2}e_{j})}{N_{2}})\chi \left( MP_{e_{k}}^{\perp }(x-vt-P_{\omega
}^{\Vert }(v^{\ast }-u^{\ast })t_{0})\right) \\
&&\chi (\frac{P_{e_{k}}(x-vt-P_{\omega }^{\Vert }(v^{\ast }-u^{\ast })t_{0})%
}{N_{2}})\chi (MP_{e_{k}}^{\perp }u^{\ast })\chi (10\frac{P_{e_{k}}(u^{\ast
}-N_{2}e_{k})}{N_{2}})dt_{0}
\end{eqnarray*}

Substitute $t_{0}$ by $\left( P_{\omega }^{\Vert }(v^{\ast }-u^{\ast })\cdot
\omega \right) t_{0}$ and $x$ by $x-vt$, we pick up a factor $\left\vert
P_{\omega }^{\Vert }(v^{\ast }-u^{\ast })\right\vert ^{-1}$. That is%
\begin{eqnarray*}
&&\int_{\mathbb{S}^{2}}d\omega \left\Vert S_{j,k}(t,x,u^{\ast },v^{\ast
})\right\Vert _{L_{x}^{\infty }} \\
&\lesssim &\chi (MP_{e_{k}}^{\perp }u^{\ast })\chi (10\frac{%
P_{e_{k}}(u^{\ast }-N_{2}e_{k})}{N_{2}})\chi (MP_{e_{j}}^{\perp }v^{\ast
})\chi (10\frac{P_{e_{j}}(v^{\ast }-N_{2}e_{j})}{N_{2}}) \\
&&\int_{\mathbb{S}^{2}}d\omega \left\vert P_{\omega }^{\Vert }(v^{\ast
}-u^{\ast })\right\vert ^{-1}\int_{0}^{\infty }\chi \left( MP_{e_{k}}^{\perp
}(x-t_{0}\omega )\right) \chi (\frac{P_{e_{k}}(x-t_{0}\omega )}{N_{2}})dt_{0}
\end{eqnarray*}%
with%
\begin{eqnarray*}
&&\int_{0}^{\infty }\chi \left( MP_{e_{k}}^{\perp }(x-t_{0}\omega )\right)
\chi (\frac{P_{e_{k}}(x-t_{0}\omega )}{N_{2}})dt_{0} \\
&\lesssim &\left\vert P_{\omega }^{\Vert }\left\{ x:\left( P_{e_{k}}^{\perp
}x<\frac{1}{M}\right) \&\&\left( P_{e_{k}}x\lesssim N_{2}\right) \right\}
\right\vert :=\left\vert P_{\omega }^{\Vert }A\right\vert
\end{eqnarray*}%
where $\left\vert P_{\omega }^{\Vert }A\right\vert $ of a set $A$ means the
length of the set A projected onto the direction $\omega $. That is,%
\begin{eqnarray*}
&&\int_{\mathbb{S}^{2}}d\omega \left\Vert S_{j,k}(t,x,u^{\ast },v^{\ast
})\right\Vert _{L_{x}^{\infty }} \\
&\lesssim &\chi (MP_{e_{k}}^{\perp }u^{\ast })\chi (10\frac{%
P_{e_{k}}(u^{\ast }-N_{2}e_{k})}{N_{2}})\chi (MP_{e_{j}}^{\perp }v^{\ast
})\chi (10\frac{P_{e_{j}}(v^{\ast }-N_{2}e_{j})}{N_{2}}) \\
&&\int_{\mathbb{S}^{2}}d\omega \left\vert P_{\omega }^{\Vert }(v^{\ast
}-u^{\ast })\right\vert ^{-1}\left\vert P_{\omega }^{\Vert }A\right\vert 
\text{.}
\end{eqnarray*}

We split into 2 cases in which Case I is when $e_{j}$ and $e_{k}$ are not
near parallel and Case II is when $e_{j}$ and $e_{k}$ are near parallel. Due
to the $u^{\ast }$ and $v^{\ast }$ cutoffs, $v^{\ast }$ is nealy parallel to 
$e_{j}$ and $u^{\ast }$ is nearly parallel to $e_{k}$.

For Case I, in which $e_{j}$ and $e_{k}$ are not near parallel, if we write $%
\theta $ to be the angle between $\omega $ and $v^{\ast }-u^{\ast }$, then%
\begin{equation*}
\left\vert P_{\omega }^{\Vert }(v^{\ast }-u^{\ast })\right\vert \sim
N_{2}\cos \theta =N_{2}\sin (\frac{\pi }{2}-\theta )\sim N_{2}(\frac{\pi }{2}%
-\theta )
\end{equation*}%
and we partition $\theta $ so that $(\frac{\pi }{2}-\theta )$ is a dyadic
number going down to the scale of $(MN_{2})^{-1}$. On the other hand, let $%
\alpha $ be the angle between $\omega $ and $e_{k}$, then 
\begin{equation*}
\left\vert P_{\omega }^{\Vert }A\right\vert \lesssim M^{-1}(\sin \alpha
)^{-1}\sim M^{-1}\alpha ^{-1}
\end{equation*}%
due to the size of the tubes and we partition $\alpha $ into dyadics going
down to the scale of $(MN_{2})^{-1}$. For a given $(\frac{\pi }{2}-\theta )$
and $\alpha $, the measure of the corresponding $\omega $-set is then 
\begin{equation*}
\alpha \min \left( \frac{\pi }{2}-\theta ,\alpha \right) \leqslant \alpha
\left( \frac{\pi }{2}-\theta \right)
\end{equation*}%
Hence,%
\begin{eqnarray*}
\int_{\mathbb{S}^{2}}d\omega \left\vert P_{\omega }^{\Vert }(v^{\ast
}-u^{\ast })\right\vert ^{-1}\left\vert P_{\omega }^{\Vert }A\right\vert
&\leqslant &\sum_{\alpha ,(\frac{\pi }{2}-\theta )}^{(MN_{2})^{-1}}\left(
N_{2}(\frac{\pi }{2}-\theta )\right) ^{-1}\frac{M^{-1}}{\alpha }\alpha
\left( \frac{\pi }{2}-\theta \right) \\
&\leqslant &\left( MN_{2}\right) ^{-1}\sum_{\alpha ,(\frac{\pi }{2}-\theta
)}^{(MN_{2})^{-1}}1 \\
&\leqslant &\left( MN_{2}\right) ^{-1}\left( \ln MN_{2}\right) ^{2}
\end{eqnarray*}%
That is, in Case I, we have%
\begin{eqnarray*}
\left\Vert D_{I}^{+}\right\Vert _{L_{v}^{1}L_{x}^{\infty }} &\leqslant &%
\frac{M^{2-2s}}{N_{2}^{4+2s}}\sum_{k}^{J}\sum_{j}^{J}\int du^{\ast }\int
dv^{\ast }\int_{\mathbb{S}^{2}}d\omega \left\Vert S_{j,k}(t,x,u^{\ast
},v^{\ast })\right\Vert _{L_{x}^{\infty }} \\
&\leqslant &\frac{M^{2-2s}}{N_{2}^{4+2s}}(MN_{2})^{4}(M^{-2}N_{2})^{2}\left(
MN_{2}\right) ^{-1}\left( \ln MN_{2}\right) ^{2} \\
&\lesssim &\left( MN_{2}\right) ^{1-2s}\left( \ln MN_{2}\right) ^{2}
\end{eqnarray*}%
where $(M^{-2}N_{2})^{2}$ and $(MN_{2})^{4}$ come from the $L_{v^{\ast
}}^{1}L_{u^{\ast }}^{1}$ integrals and the number of summands in $%
\sum_{k}^{J}\sum_{j}^{J}$.

For Case II, in which $e_{j}$ and $e_{k}$ are near parallel, we reuse the
computation in Case I but with $N_{2}$ replaced by $N_{2}\beta $ where $%
\beta $ is the angle between $e_{j}$ and $e_{k}$. For some $a,b\in \left[ 
\frac{9}{10}N_{2},\frac{11}{10}N_{2}\right] $, we can estimate%
\begin{equation*}
\left\vert v^{\ast }-u^{\ast }\right\vert ^{2}=\left\vert
ae_{j}-be_{k}\right\vert ^{2}=(a-b)^{2}+2ab(1-\cos \beta )\gtrsim
N_{2}^{2}(1-\cos \beta )\sim N_{2}^{2}\beta ^{2}
\end{equation*}%
The above also gives a reduction of the double sum $\sum_{k}^{J}\sum_{j}^{J}$%
: Fix a $e_{j}$, the summands in the $k$-sum got reduced to $\left(
MN_{2}\right) ^{2}\beta ^{2}$. Then the computation in Case I results in%
\begin{equation*}
\left\Vert D_{II}^{+}\right\Vert _{L_{v}^{1}L_{x}^{\infty }}\lesssim \left(
MN_{2}\right) ^{1-2s}\text{.}
\end{equation*}%
That is%
\begin{equation*}
\left\Vert D^{+}\right\Vert _{L_{v}^{1}L_{x}^{\infty }}\lesssim \left(
MN_{2}\right) ^{1-2s}\left( \ln MN_{2}\right) ^{2}
\end{equation*}%
which is good enough as long as $s>\frac{1}{2}$ and $\delta \leq 2(s-\frac{1%
}{2})$.

\subsection{Bilinear $Z$ norm estimates\label{sec:z-norm bilinear}}

\begin{lemma}[Bilinear $Z$ norm estimates for loss/gain operator $Q^\pm$]
\label{L:Z_bilinear} For any $f_1$, $f_2$ and any fixed $t\in \mathbb{R}$, 
\begin{equation*}
\|Q^\pm (f_1,f_2) \|_{Z} \lesssim \|f_1\|_{Z} \|f_2\|_{Z}
\end{equation*}
\end{lemma}

\begin{proof}
First we carry out the $Q^-$ estimates. For the $L_{vx}^2$ norms, 
\begin{equation*}
M\|\langle v \rangle Q^-( f_1,f_2) \|_{L_{vx}^2} \lesssim M\| \langle v
\rangle f_1 \|_{L_{vx}^2} \|f_2\|_{L_v^1L_x^\infty} \leq \|f_1\|_Z \|f_2\|_Z
\end{equation*}
Also, 
\begin{equation*}
\| \langle v \rangle Q^-( \nabla_x f_1, f_2) \|_{ L_{vx}^2} \lesssim
\|\langle v \rangle \nabla_x f_1 \|_{ L_{vx}^2} \| f_2 \|_{L_v^1L_x^\infty}
\leq \|f_1\|_Z \|f_2\|_Z
\end{equation*}
\begin{equation}  \label{E:forextra}
\| \langle v \rangle Q^-( f_1, \nabla_x f_2) \|_{ L_{vx}^2} \lesssim M
\|\langle v \rangle f_1 \|_{ L_{vx}^2} \, M^{-1} \| \nabla_x f_2\|_{
L_v^1L_x^\infty} \leq \|f_1\|_Z \|f_2\|_Z
\end{equation}
Next, for the $L_v^1L_x^\infty$ norms, 
\begin{equation*}
\| Q^-(f_1,f_2) \|_{L_v^1L_x^\infty} \lesssim \|f_1\|_{L_v^1L_x^\infty}
\|f_2\|_{L_v^1L_x^\infty} \leq \|f_1\|_Z \|f_2\|_Z
\end{equation*}
Also, 
\begin{equation*}
M^{-1} \| Q^-(\nabla_x f_1,f_2) \|_{L_v^1L_x^\infty} \leq M^{-1} \|\nabla_x
f_1 \|_{L_v^1L_x^\infty} \|f_2 \|_{L_v^1 L_x^\infty} \leq \|f_1\|_Z \|f_2\|_Z
\end{equation*}
\begin{equation*}
M^{-1} \| Q^-( f_1, \nabla_x f_2) \|_{L_v^1L_x^\infty} \leq \|f_1
\|_{L_v^1L_x^\infty} M^{-1} \|\nabla_x f_2 \|_{L_v^1 L_x^\infty} \leq
\|f_1\|_Z \|f_2\|_Z
\end{equation*}
The estimate \eqref{E:forextra} shows the need to include the terms $M\|
\langle v \rangle f \|_{L_{vx}^2}$ and $M^{-1}\|\nabla_x
f\|_{L_v^1L_x^\infty}$ in the definition of the $Z$ norm.

The proofs for $Q^+$ follow similarly but instead use the estimates in \cite[%
Theorem 2]{AC10} as was done in \S \ref{SS:542}.
\end{proof}

\subsection{Perturbation argument\label{sec:perturbative}}

We are looking to prove \eqref{E:fc_bound} for $f_{\text{c}}$ solving %
\eqref{E:fc_equation} on 
\begin{equation*}
T_*=-\delta (MN_{2})^{s-1}\ln M\leq t\leq 0
\end{equation*}
To do so, we will in fact prove bound slightly stronger than %
\eqref{E:fc_bound} in the $Z$ norm, making use of the bilinear estimates for 
$Q^\pm$ in Lemma \ref{L:Z_bilinear}, the bounds on $f_\text{a}$ in Lemma \ref%
{L:fa_size}, and the bounds on $F_{\text{err}}$ in Lemma \ref{L:Ferr_bound}.

\begin{proposition}
\label{P:norm_deflation} Given $s>\frac{1}{2}$, suppose that $f_{\text{c}}$
solves \eqref{E:fc_equation} with $f_{\text{c}}(0)=0$. Then for all $t$ such
that 
\begin{equation*}
T_{\ast }=-\delta (MN_{2})^{s-1}\ln M\leq t\leq 0
\end{equation*}%
we have the bound 
\begin{equation}
\Vert f_{\text{c}}(t)\Vert _{Z}\lesssim M^{-\delta/4}  \label{E:fc_bound2}
\end{equation}
\end{proposition}

\begin{proof}
Let the time interval $T_{\ast }\leq t\leq 0$ be partitioned as 
\begin{equation*}
T_{\ast }=T_{n}<T_{n-1}<T_{n-2}<\cdots <T_{2}<T_{1}<T_{0}=0
\end{equation*}%
where $T_{j}=-\delta j(MN_{2})^{s-1}$ and $n=\delta \ln M$. Thus, the length
of each time interval $I_{j}=[T_{j+1},T_{j}]$ is 
\begin{equation*}
|I_{j}|=\delta (MN_{2})^{s-1}
\end{equation*}%
We have 
\begin{equation*}
\partial _{t}f_{\text{c}}(t,x-tv,v)=G(t,x-tv,v)
\end{equation*}%
from which it follows that, for $t\in I_{j}=[T_{j+1},T_{j}]$, 
\begin{equation*}
f_{\text{c}}(t,x-tv,v)=f_{\text{c}}(T_{j},x-T_{j}v,v)+%
\int_{T_{j}}^{t}G(t_0,x-t_{0}v,v)\,dt_{0}
\end{equation*}%
where $f_{\text{c}}(T_{0})=0$. Replacing $x$ by $x+tv$, 
\begin{equation*}
f_{\text{c}}(t,x,v)=f_{\text{c}}(T_{j},x+(t-T_{j})v,v)+%
\int_{T_{j}}^{t}G(t_0,x+(t-t_{0})v,v)\,dt_{0}
\end{equation*}%
Applying the $Z$-norm 
\begin{align*}
\Vert f_{\text{c}}\Vert _{L_{I_{j}}^{\infty }Z}& \leq \Vert f_{\text{c}%
}(T_{j})\Vert _{Z}+\left\Vert
\int_{T_{j}}^{t}G(t_0,x+(t-t_{0})v,v)\,dt_{0}\right\Vert _{L_{I_{j}}^{\infty
}Z} \\
& \leq |I_{j}|\Vert Q^{\pm }(f_{\text{c}},f_{\text{a}})\Vert
_{L_{I_{j}}^{\infty }Z}+|I_{j}|\Vert Q^{\pm }(f_{\text{a}},f_{\text{c}%
})\Vert _{L_{I_{j}}^{\infty }Z}+|I_{j}|\Vert Q^{\pm }(f_{\text{c}},f_{\text{c%
}})\Vert _{L_{I_{j}}^{\infty }Z} \\
& \qquad +\left\Vert \int_{T_{j}}^{t}F_{\text{err}}(t_0,x+(t-t_{0})v,v)%
\,dt_{0}\right\Vert _{L_{I_{j}}^{\infty }Z}
\end{align*}%
For the terms on the first line, we apply the bilinear estimate in Lemma \ref%
{L:Z_bilinear}, and then the estimate on $\Vert f_{\text{a}}\Vert
_{L_{I_{j}}^{\infty }Z}$ from Lemma \ref{L:fa_size}. For the term on the
last line, we apply the estimate from Lemma \ref{L:Ferr_bound}. This yields
the following bound 
\begin{equation*}
\Vert f_{\text{c}}\Vert _{L_{I_{j}}^{\infty }Z}\leq \Vert f_{\text{c}%
}(T_{j})\Vert _{Z}+2\delta C\Vert f_{\text{c}}\Vert _{L_{I_{j}}^{\infty
}Z}+\delta C(MN_{2})^{s-1}\Vert f_{\text{c}}\Vert _{L_{I_{j}}^{\infty
}Z}^{2}+M^{-\delta }
\end{equation*}%
where $C$ is some absolute constant independent of $\delta $. For $\delta $
sufficiently small in terms of $C$, the terms with $\Vert f_{\text{c}}\Vert
_{L_{I_{j}}^{\infty }Z}$ on the right can be absorbed on the left yielding 
\begin{equation*}
\Vert f_{\text{c}}\Vert _{L_{I_{j}}^{\infty }Z}\leq 2\Vert f_{\text{c}%
}(T_{j})\Vert _{Z}+CM^{-\delta }
\end{equation*}%
Applying this successively for $j=0,1,\ldots $, we obtain 
\begin{equation*}
\Vert f_{\text{c}}\Vert _{L_{I_{j}}^{\infty }Z}\leq (2^{j+1}-1)CM^{-\delta }
\end{equation*}%
Evaluating this with $j=n=\delta \ln M$, we obtain the desired conclusion
that 
\begin{equation*}
\Vert f_{\text{c}}(T_{\ast })\Vert _{Z}\leq 2M^{-(1-\ln 2)\delta }\leq
M^{-\delta /4}\ll 1
\end{equation*}
\end{proof}

\begin{corollary}[norm deflation]
\label{C:norm_deflation}Given $s>1/2$ and let $s_{0}$ be given by %
\eqref{E:r_def}. There exists an exact solution $f_{\text{ex}}$ that
satisfies the same estimates as \eqref{E:fa_size_r}, specifically 
\begin{equation*}
\Vert f_{\text{ex}}(t)\Vert _{L_{v}^{2,s_{0}}H_{x}^{s_{0}}}\sim \frac{1}{\ln
M}\exp [|t|(MN_{2})^{1-s}]
\end{equation*}%
In particular, 
\begin{equation*}
\Vert f_{\text{ex}}(0)\Vert _{L_{v}^{2,s_{0}}H_{x}^{s_{0}}}\sim \frac{1}{\ln
M}\ll 1\,,\qquad \Vert f_{\text{ex}}(T_{\ast })\Vert
_{L_{v}^{2,s_{0}}H_{x}^{s_{0}}}\sim \frac{M^{\delta }}{\ln M}\gg 1
\end{equation*}
\end{corollary}

\begin{proof}
In Proposition \ref{P:norm_deflation}, we have obtained the bound %
\eqref{E:fc_bound2} for $f_{\text{c}}$ solving \eqref{E:fc_equation}.
However \eqref{E:fc_equation} is equivalent to the statement that $f_{\text{%
ex}}=f_{\text{a}}+f_{\text{c}}$ is an exact solution of Boltzmann. The
corollary just follows from the fact that the norm $%
L_{v}^{2,s_{0}}H_{x}^{s_{0}}$ is controlled by the $Z$ norm, so that $f_{%
\text{c}}$ is much smaller than $f_{\text{a}}$ on the whole time interval $%
T^{\ast }\leq t\leq 0$ in $L_{v}^{2,s_{0}}H_{x}^{s_{0}}$, and the size of $%
f_{\text{ex}}$ in this norm matches the size of $f_{\text{a}}$ in this norm.
\end{proof}

\begin{corollary}[failure of uniform continuity of the data-to-solution map]

\label{C:data_to_sol}Given $s>1/2$ and let $s_{0}$ be given by %
\eqref{E:r_def}. For each $M\gg 1$, there exists a sequence of times $%
t_{0}^{M}<0$ such that $t_{0}^{M}\nearrow 0$ and two exact solutions $f_{%
\text{ex}}^{M}(t)$, $g_{\text{ex}}^{M}(t)$ to Boltzmann on $t_{0}^{M}\leq
t\leq 0$ such that 
\begin{equation*}
\Vert f_{\text{ex}}^{M}(t_{0}^{M})\Vert _{L_{v}^{2,s_{0}}H_{x}^{s_{0}}}\sim
1\,,\qquad \Vert g_{\text{ex}}^{M}(t_{0}^{M})\Vert
_{L_{v}^{2,s_{0}}H_{x}^{s_{0}}}\sim 1
\end{equation*}%
with initial closeness at $t=t_{0}^{M}$ 
\begin{equation*}
\Vert f_{\text{ex}}^{M}(t_{0}^{M})-g_{\text{ex}}^{M}(t_{0}^{M})\Vert
_{L_{v}^{2,s_{0}}H_{x}^{s_{0}}}\leq \frac{1}{\ln M}\ll 1
\end{equation*}%
and full separation at $t=0$ 
\begin{equation*}
\Vert f_{\text{ex}}^{M}(0)-g_{\text{ex}}^{M}(0)\Vert
_{L_{v}^{2,s_{0}}H_{x}^{s_{0}}}\sim 1
\end{equation*}
\end{corollary}

\begin{proof}
For any $M\gg 1$, let $f_{\text{ex}}(t)$ be the solution given in Corollary %
\ref{C:norm_deflation}. Let $t_{0}$ be the time $T_{\ast }\leq t_{0}\leq 0$
at which $\Vert f_{\text{r}}(t_{0})\Vert _{L_{v}^{2,s_{0}}H_{x}^{s_{0}}}=1$.
Let $g_{\text{ex}}$ be the exact solution to Boltzmann with $g_{\text{ex}%
}(t_{0})=f_{\text{r}}(t_{0})$. Applying the same methods to approximate $g_{%
\text{ex}}(t)$, we obtain that for all $t_{0}\leq t\leq 0$, 
\begin{equation*}
g_{\text{ex}}(t)=f_{\text{r}}(t_{0})+g_{\text{c}}(t)
\end{equation*}%
where 
\begin{equation*}
\Vert g_{\text{c}}\Vert _{L_{t_{0}\leq t\leq 0}^{\infty }Z}\leq M^{-\delta }
\end{equation*}%
(Conceptually, this just results from the fact that on this short time
interval, $f_{\text{r}}$ is nearly stationary and moveover has small
self-interaction through the gain and loss terms, and thus the constant
function $f_{\text{r}}(t_{0})$ is a good approximation to an exact solution
to Boltzmann. )

Thus we have two solutions with the decompositions 
\begin{equation*}
f_{\text{ex}}(t)=f_{\text{r}}(t)+f_{\text{b}}(t)+f_{\text{c}}(t)
\end{equation*}%
\begin{equation*}
g_{\text{ex}}(t)=f_{\text{r}}(t_{0})+g_{c}(t)
\end{equation*}%
which gives 
\begin{equation*}
f_{\text{ex}}(t)-g_{\text{ex}}(t)=(f_{\text{r}}(t)-f_{\text{r}}(t_{0}))+f_{%
\text{b}}(t)+f_{\text{c}}(t)-g_{\text{c}}(t)
\end{equation*}%
For all $t$, 
\begin{equation*}
\Vert f_{\text{b}}(t)\Vert _{L_{v}^{2,s_{0}}H_{x}^{s_{0}}}\sim
(MN_{2})^{s_{0}-s}\sim \frac{1}{(\ln M)^{1+\mu }}
\end{equation*}%
where $N_{2}=M^{\mu }$ (recall we required $\mu \geq \delta $). Moreover, 
\begin{equation*}
\Vert f_{\text{c}}(t)\Vert _{L_{v}^{2,s_{0}}H_{x}^{s_{0}}}\leq \Vert f_{%
\text{c}}(t)\Vert _{Z}\lesssim M^{-\delta }\exp [|t|(MN_{2})^{s-1}]\leq
M^{-\delta (1-\ln 2)}
\end{equation*}%
just like in the end of the proof of Proposition \ref{P:norm_deflation}, and 
\begin{equation*}
\Vert g_{\text{c}}(t)\Vert _{L_{v}^{2,s_{0}}H_{x}^{s_{0}}}\leq M^{-\delta }
\end{equation*}%
Thus 
\begin{equation*}
\Vert f_{\text{ex}}(t_{0})-g_{\text{ex}}(t_{0})\Vert
_{L_{v}^{2,s_{0}}H_{x}^{s_{0}}}\sim \frac{1}{(\ln M)^{1+\mu }}
\end{equation*}%
and 
\begin{equation*}
\Vert f_{\text{ex}}(0)-g_{\text{ex}}(0)\Vert
_{L_{v}^{2,s_{0}}H_{x}^{s_{0}}}\sim \Vert f_{\text{r}}(0)-f_{\text{r}%
}(t_{0})\Vert _{L_{v}^{2,s_{0}}H_{x}^{s_{0}}}\sim 1
\end{equation*}
\end{proof}

\appendix 

\section{Remarks on scaling and Strichartz estimates}

\label{A:scaling}

In this appendix, we give some remarks on the scaling of (\ref{E:boltz1})
and its effects. Let $f_{\lambda }$ be as in (\ref{E:scaling condition}),
then 
\begin{equation*}
\Vert |\nabla _{x}|^{s}|v|^{r}f_{\lambda }\Vert _{L_{xv}^{2}}=\Vert |\nabla
_{x}|^{s}|v|^{r}f\Vert _{L_{xv}^{2}}
\end{equation*}%
if and only if 
\begin{equation}
\alpha s-\beta r=\frac{1}{2}(\alpha -\beta ).  \label{E:scal}
\end{equation}%
At criticality, \eqref{E:scal} must hold for all $\alpha ,\beta \in \mathbb{R%
}$. Taking $\alpha =\beta =1$, \eqref{E:scal} implies $s=r$. Setting $s=r$
in \eqref{E:scal}, we obtain $s(\alpha -\beta )=\frac{1}{2}(\alpha -\beta )$%
, from which we conclude that $s=r=\frac{1}{2}$. Hence, the norm is defined
as in (\ref{e:sobolev norm}) and we call (\ref{E:boltz1}) $\dot{H}_{x\xi }^{%
\frac{1}{2}}$-invariant. Moreover, the case $s=r$ allows a well-defined
notion of \emph{subcriticality}: since 
\begin{equation}
\Vert |\nabla _{x}|^{s}|v|^{s}f_{\lambda }\Vert _{L_{xv}^{2}}=\lambda ^{(s-%
\frac{1}{2})(\alpha -\beta )}\Vert |\nabla _{x}|^{s}|v|^{s}f\Vert
_{L_{xv}^{2}}  \label{eqn:norm scales}
\end{equation}%
a large data problem over a short time can be converted into a small data
problem over a long time when $s=r>\frac{1}{2}$.

Such a scaling property also effects the Strichartz estimates (\ref%
{e:Strichartz}). Suppose $\tilde{f}$ solves \eqref{E:01}. Then an estimate
of the type 
\begin{equation}
\Vert \tilde{f}\Vert _{L_{t\in I}^{q}L_{\xi }^{r}L_{x}^{p}}\lesssim \Vert 
\tilde{f}|_{t=0}\Vert _{L_{x\xi }^{2}}  \label{E:05}
\end{equation}%
is called a Strichartz estimate, where either $I=(-\infty ,+\infty )$, or $I$
is some fixed subinterval of time. A \emph{necessary condition} for such an
estimate (regardless of $I$) is that $r=p$. This follows from the fact that
if $f$ solves \eqref{E:01}, then 
\begin{equation*}
\tilde{g}(x,\xi ,t)=\tilde{f}(\lambda x,\lambda ^{-1}\xi ,t)
\end{equation*}%
solves \eqref{E:01} on the same time interval. If $p\neq r$, then we obtain
a contradiction to \eqref{E:05} by either sending $\lambda \rightarrow 0$ or 
$\lambda \rightarrow \infty $.

The estimate (\ref{E:05}) is valid for $p=r$ when the scaling condition in (%
\ref{e:Strichartz}) is met. The case $q=\infty $, $p=2$ is simply the fact
that the equation preserves the $L_{x\xi }^{2}$ norm. The case $q=2$, $p=3$
follows from the endpoint argument of Keel \& Tao \cite{KT98}, since the
corresponding linear propagator $e^{it\nabla _{x}\cdot \nabla _{\xi }}$
satisfies the dispersive estimate 
\begin{equation*}
\Vert e^{it\nabla _{x}\cdot \nabla _{\xi }}\phi \Vert _{L_{x\xi }^{\infty
}}\lesssim t^{-3}\Vert \phi \Vert _{L_{x\xi }^{1}}.
\end{equation*}%
We note that the Strichartz estimates (\ref{e:Strichartz}) are not the same
as those labeled in some literature as \textquotedblleft Strichartz estimate
for the kinetic transport equation\textquotedblright\ -- see \cite{BBGL14,
BP01, Ov11} for examples.

\section{Proof of the asymmetric bilinear estimate for the loss term}

\label{B:BilinearProof}

In order to give a shorter proof of \eqref{E:loss-1} in Theorem \ref{T:loss}%
, we shall prove instead an estimate for the $X_{0,0}$ norm (which is
actually \emph{larger} than $X_{0,-\frac{1}{2}+}$), although the bilinear
gain factor is not fully realized in every case. Specifically, we obtain 
\begin{equation}
\Vert \theta (t)\tilde{Q}^{-}(\tilde{f},\tilde{g})\Vert _{X_{0,0}}\lesssim
\min (M_{1},M_{2})N_{2}B_{M_{1},M_{2},N,N_{2}}\Vert \tilde{f}\Vert _{X_{0,%
\frac{1}{2}+}}\Vert \tilde{g}\Vert _{X_{0,\frac{1}{2}+}}  \label{E:loss-1b}
\end{equation}%
with the bilinear gain factor 
\begin{equation*}
B_{M_{1},M_{2},N,N_{2}}=%
\begin{cases}
1 & \text{if }M_{2}\leq M_{1}\,,\;N\leq N_{2} \\ 
(M_{1}/M_{2})^{1/2} & M_{1}\ll M_{2}\sim M\,,\;N\leq N_{2} \\ 
(N_{2}/N)^{1/2} & M_{2}\leq M_{1}\sim M\,,\;N_{2}\ll N \\ 
(\frac{M_{1}}{M_{2}})^{1/4}(\frac{N_{2}}{N})^{1/4} & M_{1}\ll
M_{2}\,,\;N_{2}\ll N%
\end{cases}%
\end{equation*}%
Note that in the statement of Theorem \ref{T:loss}, the estimate claimed for
the $X_{0,-\frac{1}{2}+}$ norm in \eqref{E:loss-1} for the case $M_{1}\ll
M_{2}$, $N_{2}\ll N$ has the stronger bilinear gain factor: 
\begin{equation*}
B_{M_{1},M_{2}}B_{N,N_{2}}=(\frac{M_{1}}{M_{2}})^{1/2}(\frac{N_{2}}{N})^{1/2}
\end{equation*}%
Otherwise, the estimate \eqref{E:loss-1b} matches \eqref{E:loss-1} for the $%
X_{0,-\frac{1}{2}+}$ norm. The main point here is the fact that the bilinear
gain factor is asymmetric in $f$ and $g$. One could bring in the cone-washer
decompostion like \S \ref{sec:sharpness} to reach the stronger bilinear gain
factor, but that could be too much for one estimate not used in this paper.

To prove \eqref{E:loss-1b}, it suffices, by a standard reduction, to assume
that $\tilde{f}=e^{it\nabla _{x}\cdot \nabla _{\xi }}\tilde{\phi}$ and $%
\tilde{g}=e^{it\nabla _{x}\cdot \nabla _{\xi }}\tilde{\psi}$. Then 
\begin{equation*}
\Vert \tilde{Q}^{-}(\tilde{f},\tilde{g})\Vert _{X_{0,0}}=\Vert e^{-it\nabla
_{x}\cdot \nabla _{\xi }}\tilde{Q}^{-}(e^{it\nabla _{x}\cdot \nabla _{\xi }}%
\tilde{\phi},e^{it\nabla _{x}\cdot \nabla _{\xi }}\tilde{\psi})\Vert
_{L_{x\xi t}^{2}}
\end{equation*}%
Passing to the Fourier side $(x,\xi )\mapsto (\eta ,v)$, and using a dual
pairing with a function $\hat{\zeta}(\eta ,v,\tau )\in L_{\eta v\tau }^{2}$,
it suffices to bound\footnote{%
Notice that, such a process does not require the Fourier transform of the
collision kernel as we estimate in the $v$-side.} 
\begin{equation*}
\int_{\eta ,\eta _{2},v,v_{2}}\hat{\phi}(\eta -\eta _{2},v)\hat{\psi}(\eta
_{2},v_{2})\hat{\zeta}(\eta ,v,\eta _{2}\cdot (v_{2}-v))\,d\eta \,d\eta
_{2}\,dv\,dv_{2}
\end{equation*}%
\begin{equation*}
\lesssim \min (M_{1},M_{2})N_{2}B_{M_{1},M_{2}}B_{N,N_{2}}\Vert \hat{\phi}%
\Vert _{L^{2}}\Vert \hat{\psi}\Vert _{L^{2}}\Vert \hat{\zeta}\Vert _{L^{2}}
\end{equation*}%
where we can assume each factor $\hat{\phi}\geq 0$, $\hat{\psi}\geq 0$, $%
\hat{\zeta}\geq 0$. Replacing $w_{2}=v_{2}-v$, 
\begin{equation*}
=\int_{\eta ,\eta _{2},v,w_{2}}\hat{\phi}(\eta -\eta _{2},v)\hat{\psi}(\eta
_{2},w_{2}+v)\hat{\zeta}(\eta ,v,\eta _{2}\cdot w_{2})\,d\eta \,d\eta
_{2}\,dv\,dw_{2}
\end{equation*}%
We use superscripts to denote components, for example $\eta _{2}=(\eta
_{2}^{1},\eta _{2}^{2},\eta _{2}^{3})$.

\bigskip

\noindent \emph{Case 1}. $M_2 \leq M_1$. Assume that 
\begin{equation*}
|\eta_2^3|= \max(|\eta_2^1|, |\eta_2^2|, |\eta_2^3|) \sim |\eta_2| \sim M_2
\end{equation*}
(the other two cases are similar). Move the $v$, $\eta_2$, $w_2^1$, $w_2^2$
integration to the outside and the $\eta$, $w_2^3$ integration on the
inside. Cauchy-Schwarz in $\eta$, $w_2^3$ to obtain 
\begin{equation*}
\lesssim \int_{v,\eta_2,w_2^1,w_2^2} \| \hat \phi(\eta-\eta_2,v) \hat \psi(
\eta_2, w_2+v) \|_{L_{\eta w_2^3}^2} \| \hat\zeta( \eta, v, \eta_2\cdot w_2)
\|_{L_{\eta w_2^3}^2} \, d\eta_2 \, dw_2^1 \, dw_2^2 \, dv
\end{equation*}
For the $\hat \zeta$ term, change variable from $w_2^3$ to $\tau = \eta_2
\cdot w_2$ (here, $\eta_2$ and $w_2^1$, $w_2^2$ can be regarded as fixed,
since they are in the outside integration). The change of differential is $%
d\tau = |\eta_2^3| dw_2^3$, and since $|\eta_2^3|\sim M_2$, 
\begin{equation*}
\lesssim M_2^{-1/2} \int_v \left( \int_{\eta_2,w_2^1,w_2^2} \| \hat \psi(
\eta_2, w_2+v) \|_{L_{w_2^3}^2} \, d\eta_2 \, dw_2^1 \, dw_2^2 \right) \|
\hat \phi(\eta_1,v)\|_{L^2_{\eta_1}} \| \hat \zeta(\eta,v,\tau)
\|_{L^2_{\eta \tau}} dv
\end{equation*}
For fixed $v$, the $w_2^1$, $w_2^2$ integrations are confined to sets of
width $N_2$ (even if $N\gg N_2$), due to the fact that $w_2=v+v_2$. In the
inner integral, Cauchy-Schwarz in $\eta_2, w_2^1, w_2^2$ with the whole
integrand in one factor and $1$ in the other to obtain the factor $%
M_2^{3/2}N_2$ coming from the support, 
\begin{equation*}
\lesssim M_2^{-1/2}M_2^{3/2} N_2 \| \hat\psi\|_{L^2} \int_v \| \hat
\phi(\eta_1,v)\|_{L^2_{\eta_1}} \| \hat \zeta(\eta,v,\tau) \|_{L^2_{\eta
\tau}} dv
\end{equation*}
Cauchy-Schwarz in $v$ to obtain 
\begin{equation*}
\lesssim M_2N_2 \| \hat \phi \|_{L^2} \| \hat \psi \|_{L^2} \| \hat \zeta
\|_{L^2}
\end{equation*}

\bigskip

\noindent \emph{Case 2}. $M_1 \ll M_2\sim M$. Divide the $\eta$ and $\eta_2$
space into cubes of size $M_1$. Once one of these cubes in $\eta_2$ space is
selected, the inner $\eta$ integral is confined to $O(1)$ cubes. We can
therefore carry the Case 1 argument above out and it will yield instead the
factor $M_2^{-1/2}M_1^{3/2} N_2 $. We finish with Cauchy-Schwarz over the
cube partition (which incurs no loss). 
\begin{equation*}
\lesssim M_2^{-1/2}M_1^{3/2} N_2 \| \hat \phi \|_{L^2} \| \hat \psi \|_{L^2}
\| \hat \zeta \|_{L^2}
\end{equation*}

\bigskip

\noindent \emph{Case 3}. $M_2 \leq M_1$ and $N \gg N_2$. Without loss we may
assume that 
\begin{equation*}
|w_2^3| = \max(|w_2^1|, |w_2^2|, |w_2^3|) \sim N
\end{equation*}
(the other two cases are similar). Move the $v, w_2, \eta_2^1, \eta_2^2$
integration to the outside, bring the $\eta, \eta_2^3$ integration to the
inside, and Cauchy-Schwarz in $\eta, \eta_2^3$ to obtain 
\begin{equation*}
\lesssim \int_{v, w_2, \eta_2^1, \eta_2^2 } \| \hat \phi(\eta-\eta_2,v) \hat
\psi( \eta_2, w_2+v) \|_{L_{\eta \eta_2^3}^2} \| \hat\zeta( \eta, v,
\eta_2\cdot w_2) \|_{L_{\eta \eta_2^3}^2} \, dv \, dw_2 \, d\eta_2^1 \,
d\eta_2^2
\end{equation*}
For the $\hat \zeta$ term, change variable from $\eta_2^3$ to $\tau =
\eta_2\cdot w_2$, which has differential conversion $d\tau = |w_2^3|
d\eta_2^3$ ($w_2$ is fixed, since the term is inside the $w_2$ integration).
Since $|w_2^3|\sim N$, 
\begin{equation*}
\lesssim N^{-1/2} \int_{v, w_2, \eta_2^1, \eta_2^2 } \| \hat
\phi(\eta_1,v)\|_{L^2_{\eta_1}} \| \hat \psi( \eta_2, w_2+v)
\|_{L_{\eta_2^3}^2} \| \hat \zeta(\eta,v,\tau)\|_{L^2_{\eta\tau}} \, dv \,
dw_2 \, d\eta_2^1 \, d\eta_2^2
\end{equation*}
Converting back from $w_2$ to $v_2=w_2+v$, the remaining integrals split: 
\begin{equation*}
= N^{-1/2} \int_v \| \hat \phi(\eta_1,v)\|_{L^2_{\eta_1}} \| \hat
\zeta(\eta,v,\tau)\|_{L^2_{\eta\tau}} \, dv \int_{v_2, \eta_2^1, \eta_2^2}
\| \hat \psi( \eta_2, v_2) \|_{L_{\eta_2^3}^2} \, dv_2 \, d\eta_2^1 \,
d\eta_2^2
\end{equation*}
In the first integral, Cauchy-Schwarz in $v$, and in the second integral,
Cauchy-Schwarz with the whole integrand in one part and $1$ in the other to
pick up the support factor $N_2^{3/2}M_2$. 
\begin{equation*}
\lesssim N^{-1/2} N_2^{3/2} M_2 \| \hat \phi \|_{L^2} \| \hat \psi \|_{L^2}
\| \hat \zeta \|_{L^2}
\end{equation*}

\bigskip

\noindent \emph{Case 4}. $M_{1}\ll M_{2}$ and $N\gg N_{2}$. Divide the $\eta 
$ and $\eta _{2}$ space into cubes of size $M_{1}$. Once one of these cubes
in $\eta _{2}$ space is selected, the inner $\eta $ integral is confined to $%
O(1)$ cubes. We can therefore carry the Case 3 argument above out and it
will yield instead the factor $N^{-1/2}N_{2}^{3/2}M_{1}$. We finish with
Cauchy-Schwarz over the cube partition (which incurs no loss). 
\begin{equation*}
\lesssim N^{-1/2}N_{2}^{3/2}M_{1}\Vert \hat{\phi}\Vert _{L^{2}}\Vert \hat{%
\psi}\Vert _{L^{2}}\Vert \hat{\zeta}\Vert _{L^{2}}
\end{equation*}%
Alternatively, we can appeal to Case 2 to obtain 
\begin{equation*}
\lesssim M_{2}^{-1/2}M_{1}^{3/2}N_{2}\Vert \hat{\phi}\Vert _{L^{2}}\Vert 
\hat{\psi}\Vert _{L^{2}}\Vert \hat{\zeta}\Vert _{L^{2}}
\end{equation*}%
Taking the average gives 
\begin{equation*}
\lesssim (\frac{M_{1}}{M_{2}})^{1/4}(\frac{N_{2}}{N})^{1/4}M_{1}N_{2}\Vert 
\hat{\phi}\Vert _{L^{2}}\Vert \hat{\psi}\Vert _{L^{2}}\Vert \hat{\zeta}\Vert
_{L^{2}}
\end{equation*}

\end{document}